\documentclass[12pt]{article}

\usepackage{amsmath}
\usepackage{amssymb}
\usepackage{amsthm}
\usepackage{bbm,bm}
\usepackage{mathrsfs, color}
\usepackage{subfigure}
\usepackage{graphicx}
\usepackage{graphics}
\usepackage[pdftex,pagebackref,letterpaper=true,colorlinks=true,pdfpagemode=none,urlcolor=blue,linkcolor=blue,
citecolor=blue,pdfstartview=FitH,plainpages=true]{hyperref}
\usepackage[comma,authoryear]{natbib}

\newcommand{\eps}[0]{\varepsilon}

\newcommand{\Cov}[0]{\text{Cov}}

\newcommand{\diag}[0]{\text{diag}}

\newcommand{\calF}[0]{\mathcal{F}}

\newcommand{\E}[0]{\mathbb{E}}
\renewcommand{\P}[0]{\mathbb{P}}

\newcommand{\Prob}[0]{\mathbb{P}}

\newcommand{\lp}[0]{\left(}
\newcommand{\rp}[0]{\right)}
\newcommand{\lal}[0]{\left|}
\newcommand{\ral}[0]{\right|}
\newcommand{\bigindex}[1]{\uppercase\expandafter{\romannumeral#1}}

\theoremstyle{plain}
\newtheorem{thm}{Theorem}[section]
\newtheorem{lem}[thm]{Lemma}

\newtheorem{cor}[thm]{Corollary}
\newtheorem{assumption}[thm]{Assumption}

\theoremstyle{definition}

\theoremstyle{remark}
\newtheorem{rmk}{Remark}

\long\def\hid#1*/{}  
\definecolor{Blue}{rgb}{0,0,1}
\definecolor{Red}{rgb}{1,0,0}
\begin{document}
\begin{center}
\large{\bf Gaussian Approximation for High Dimensional Time Series}
\end{center}

\begin{center}
{By Danna Zhang and Wei Biao Wu}
\end{center}

\begin{center}
{Department of Statistics, University of Chicago}
\end{center}

\begin{center}
{\today}
\end{center}

\begin{abstract}
We consider the problem of approximating sums of high-dimensional stationary time series by Gaussian vectors, using the framework of functional dependence measure. The validity of the Gaussian approximation depends on the sample size $n$, the dimension $p$, the moment condition and the dependence of the underlying processes. We also consider an estimator for long-run covariance matrices and study its convergence properties. Our results allow constructing simultaneous confidence intervals for mean vectors of high-dimensional time series with asymptotically correct coverage probabilities. A Gaussian multiplier bootstrap method is proposed. A simulation study indicates the quality of Gaussian approximation with different $n$, $p$ under different moment and dependence conditions.
\end{abstract}

\section{Introduction}
\label{sec:introduction}
During the past decade, there has been a significant development on high-dimensional data analysis with applications in many fields. In this paper we shall consider simultaneous inference for mean vectors of high-dimensional stationary processes, so that one can perform family-wise multiple testing or construct simultaneous confidence intervals, an important problem in the analysis of spatial-temporal processes. To fix the idea, let $X_i$ be a stationary process in $\mathbb R^p$ with mean $\mu = (\mu_1, \ldots, \mu_p)^\top$ and finite second moment in the sense that $\E (X_i^\top X_i) < \infty$. In the scalar case in which $p = 1$ or when $p$ is fixed, under suitable weak dependence conditions, we can have the central limit theorem (CLT)
\begin{eqnarray}
\label{eq:15140307}
{1\over \sqrt n} \sum_{i=1}^n (X_i-\mu) \Rightarrow N(0, \Sigma), \mbox{ where }
  \Sigma = \sum_{k=-\infty}^\infty \E ( (X_0-\mu) (X_k-\mu)^\top).
\end{eqnarray}
See, for example, \citet{rosenblatt1956central}, \citet{ibragimov1971independent}, \citet{Wu2005}, \citet{dedecker2007weak}
and \citet{bradley2007introduction} among others. In the high dimension case in which $p$ can also diverge to infinity, \citet{portnoy1986central} showed that the central limit theorem can fail for i.i.d. random vectors if $\sqrt n = o(p)$. In this paper we shall consider an alternative form: Gaussian approximation for the largest entry of the sample mean vector  $\bar X_n = n^{-1} \sum_{i=1}^n X_i$. For a vector $v = (v_1, \ldots, v_p)^\top$, let $|v|_\infty = \max_{j\le p} |v_j|$. Specifically, our primary goal is to establish the Gaussian Approximation (GA) in $\mathbb R^p$
\begin{eqnarray}\label{eq:1507172206}
 \sup_{u \ge 0} |\P( \sqrt n |\bar X_n-\mu|_\infty \ge u) - \P(|Z|_\infty \ge u)| \to 0,
\end{eqnarray}
where both $n, p \to \infty$. Here the Gaussian vector $Z = (Z_1, \ldots, Z_p)^\top \sim N(0, \Sigma)$.   \citet{chernozhukov2013} studied the Gaussian approximation for independent random vectors. There has been limited research on high-dimensional inference under dependence. The associated statistical inference becomes considerably more challenging since the autocovariances with all lags should be considered. \citet{zhang2014} extended the Gaussian approximation in \citet{chernozhukov2013} to very weakly dependent random vectors which satisfy a uniform geometric moment contraction condition. The latter condition is also adopted in \citet{csw2015} for self-normalized sums. \citet{chernozhukov2013testing} did a similar extension to strong mixing random vectors. Here we shall establish (\ref{eq:J021209}) for a wide class of high-dimensional stationary process under suitable conditions on the magnitudes of $p$, $n$, and the mild dependence conditions on the process $(X_i)$.

In Section \ref{sec:high-dimentional time series} we shall introduce the framework of high-dimensional time series and some concepts about functional and predictive dependence measures that are useful for establishing an asymptotic theory.  The main result for Gaussian approximation of the normalized mean vector and the choice of the normalization matrix is established in Section \ref{sec:Gaussian approximation}. Depending on the moment and the dependence conditions, both high dimension and ultra high dimension cases are discussed.

To perform statistical inference based on (\ref{eq:J021209}), one needs to estimate the long-run covariance matrix $\Sigma$. The latter problem has been extensively studied in the scalar case; see \citet{politis1999subsampling}, \citet{buhlmann2002}, \citet{lahiri2003resampling}, \citet{Alexopoulos2004}, among others. In Section \ref{sec:estimation of long-run variance} we study the batched-mean estimate of long-run covariance matrices and derive a large deviation result about quadratic forms of stationary processes. The latter tail probabilities inequalities allow dependent and/or non-sub-Gaussian processes under mild conditions, which is expected to be useful in other high-dimensional inference problems for dependent vectors. The consistency of the batched-mean estimate ensures the validity of the normalized Gaussian multiplier bootstrap method.

We provide in Section \ref{sec:probability inequalities} some sharp inequalities for tail probabilities for dependent processes in both polynomial tail and exponential tail cases. Part of the proof are relegated to Section \ref{sec:proof}. 

We now introduce some notation. For a random variable $X$ and $q > 0$, we write $X \in \mathcal{L}^q$ if $\|X\|_{q} := (\E |X|^q)^{1/q} < \infty$, and for a vector $v = (v_1, \ldots, v_p)^\top$, let the norm-$s$ length $|v|_s = ( \sum_{j=1}^p |v_j|^s)^{1/s}$, $s \ge 1$. Write the $p \times p$ identity matrix as $\text{Id}_p$. For two real numbers, set $x \vee y =\max(x, y)$ and $x \wedge y =\min(x, y)$. For two sequences of positive numbers $(a_n)$ and $(b_n)$, we write $a_n \asymp b_n$ (resp. $a_n \lesssim b_n$ or $a_n \ll b_n$) if there exists some constant $C > 0$ such that $C^{-1} \leq a_n/b_n \leq C$ (resp. $a_n/b_n \leq C$ or $a_n/b_n \to 0$) for all large $n$. We use $C, C_1, C_2, \cdots$ to denote positive constants whose values may differ from place to place. A constant with a symbolic subscript is used to emphasize the dependence of the value on the subscript. Throughout the paper, we assume $p=p_n \rightarrow \infty$ as $n\rightarrow \infty$.

\section{High-dimensional Time Series}
\label{sec:high-dimentional time series}
Let $\eps_i, i \in \mathbb{Z}$, be i.i.d. random variables and $\calF^{i}=(\ldots, \eps_{i-1}, \eps_i)$; let $({X}_i)$ be a stationary process taking values in $\mathbb{R}^p$ that assumes the form
\begin{equation}
\label{highdimensionrepresentation}
X_i = (X_{i1}, X_{i2}, \ldots, X_{ip})^\top=G(\calF^{i}),
\end{equation}
where $G(\cdot)=(g_1(\cdot), \ldots, g_p(\cdot))^\top$ is an $\mathbb{R}^p$-valued measurable function such that $X_i$ is well-defined. In the scalar case with $p = 1$, (\ref{highdimensionrepresentation}) allows a very general class of stationary processes (cf. \citet{wiener1958nonlinear}, \citet{rosenblatt1971markov}, \citet{priestley1988non}, \citet{tong1990non}, \citet{Wu2005}, \citet{tsay2005analysis}, \citet{wu2011asymptotic}). It includes linear processes as well as a large class of nonlinear time series models. Within this framework, $(\eps_i)$ can be viewed as independent inputs of a physical system and all the dependences among the outputs $(X_i)$ result from the underlying data-generating mechanism $G(\cdot)$. The function $g_j(\cdot)$, $1 \leq j \leq p$, is the $j$-th coordinate projection of $G(\cdot)$. Unless otherwise specified, assume throughout the paper that $\E X_i=0$ and $\max_{j \le p} \| X_{ij} \|_q < \infty $ for some $q \geq 2$. Let $\Gamma(k) = (\gamma_{ij}(k))_{i, j=1}^p = \E (X_i X_{i+k}^\top)$ be the autocovariance matrix and recall the long-run covariance matrix
\begin{eqnarray}
\label{eq:J13426p}
\Sigma=(\sigma_{i j})_{i, j=1}^p=\sum_{k=-\infty}^\infty \Gamma(k)
\end{eqnarray}
if it exists. Note that $\sigma_{jj}=\sum_{k=-\infty}^\infty \gamma_{jj}(k)$, $1 \leq j \leq p$, is the long-run variance of the component process $X_{\cdot j} = (X_{i j})_{i \in \mathbb{Z}}$. For the latter process, following \citet{Wu2005} we define respectively the functional dependence and the predictive dependence measure
\begin{eqnarray}\label{eq:1507140845}
\delta_{i, q, j} &=& \| X_{ij}-X_{ij, \{0\}} \|_q=\| X_{ij}-g_j(\calF^{i, \{0\}}) \|_q, \cr
\theta_{i, q, j}&=&\| \E(X_{ij}|\calF^{0})-\E(X_{ij}|\calF^{-1})\|_q = \| \mathcal{P}^0 X_{ij} \|_q,\cr
\theta'_{i, q, j}&=&\|\E(X_{ij}|\calF_{0}^i)-\E(X_{ij}|\calF_{1}^i)\|_q=\|\mathcal{P}_0 X_{ij}\|_q,
\end{eqnarray}
where $\mathcal{F}^{i, \{j\}}=(\ldots, \varepsilon_{j-1}, \varepsilon_{j}', \varepsilon_{j+1}, \ldots, \varepsilon_{i})$ is a coupled version of $\mathcal{F}^{i}$ with $\varepsilon_{j}$ in $\mathcal{F}^i$ replaced by $\varepsilon_{j}'$, and $\varepsilon_k, \varepsilon_l'$, $k, l \in \mathbb{Z}$, are i.i.d. random variables, $\calF_i^j=(\eps_i, \eps_{i+1}, \ldots, \eps_j)$ and $\calF_{i}=(\eps_i, \eps_{i+1}, \ldots)$. Note that $\mathcal{F}^{i, \{j\}}= \mathcal{F}^{i}$ if $j > i$. To account for the dependence in the process $X_{\cdot j}$, we define the dependence adjusted norm
\begin{eqnarray}
\label{eq:2015081932}
\|X_{\cdot j} \|_{q, \alpha} = \sup_{m \geq 0} (m+1)^\alpha \Delta_{m, q, j},\,  \alpha \geq 0,
 \mbox{ where }  \Delta_{m, q, j}= \sum_{i=m}^\infty \delta_{i, q, j}, \,  m \geq 0.
\end{eqnarray}
Due to the dependence, it may happen that $\| X_{i j} \|_q < \infty$ while $\|X_{\cdot j} \|_{q, \alpha} = \infty$. Elementary calculations show that, if $X_{i j}, i \in \mathbb{Z}$, are i.i.d., then $\| X_{i j} \|_q \le \|X_{\cdot j} \|_{q, 0} \le 2 \| X_{i j} \|_q$, suggesting that the dependence adjusted norm is equivalent to the classical $L^q$ norm.

To account for high-dimensionality, we define
\begin{eqnarray*}
\Psi_{q, \alpha}= \max_{1 \leq j \leq p}\|X_{\cdot j}\|_{q, \alpha}
 \mbox{ and }
\Upsilon_{q, \alpha}=\left(\sum_{j=1}^p \|X_{\cdot j}\|_{q, \alpha}^q \right)^{1/q},
\end{eqnarray*}
which can be interpreted as the uniform and the overall dependence adjusted norms of $(X_i)_{i \in \mathbb{Z}}$, respectively.  The form (\ref{highdimensionrepresentation}) and its associated dependence measures provide a convenient framework for studying high-dimensional time series. \citet{chen2013} and \citet{zhang2014} considered some special cases:  the former paper requires that  $\max_{1 \leq j \leq p}\|X_{\cdot j}\|_{q, \alpha} \leq C$ while the latter imposes the stronger geometric moment contraction condition $\max_{1 \leq j \leq p}\Delta_{m, q, j} \le C \rho^m$ with $\rho \in (0, 1)$, and in both cases the constant $C$ does not depend on $p$. Those assumptions can be fairly restrictive. In this paper $\Psi_{q, \alpha}$ can be unbounded in $p$. Additionally, we define the $\mathcal{L}^{\infty}$ functional dependence measure and its corresponding dependence adjusted norm for the $p$-dimensional stationary process $(X_i)$
\begin{eqnarray*}
&&\omega_{i, q} = \| |X_i - X_{i, \{0\}}|_{\infty} \|_q;\\
&& \| |X_{\cdot}|_{\infty} \|_{q, \alpha}=\sup_{m \geq 0} (m+1)^\alpha \Omega_{m,q}, \alpha\geq 0, \text{ where } \Omega_{m,q}=\sum_{i=m}^\infty \omega_{i,q}, m \geq 0.
\end{eqnarray*}
Clearly, we have $\Psi_{q, \alpha} \leq \||X_{\cdot}|_\infty\|_{q, \alpha} \leq \Upsilon_{q, \alpha}$.

\section{Gaussian Approximations}
\label{sec:Gaussian approximation}

In this section we shall present main results on Gaussian approximations. Theorem \ref{th:1507140259} concerns the finite polynomial moment case with both weaker and stronger temporal dependence. Consequently the dimension $p$ allowed can be at most a power of $n$. If the underlying process has finite dependence-adjusted sub-exponential norms, Theorem \ref{th:1507140325} asserts that an ultra-high dimension $p$ can be allowed. Theorem \ref{maintheorem} in Section \ref{sec:1507151006} provides a convergence rate of the Gaussian approximation.

Recall (\ref{eq:J13426p}) for the long-run covariance matrix $\Sigma$. Let $\Sigma_0 = {\rm diag}(\Sigma)$ be the diagonal matrix of $\Sigma$, and $D_0 = {\rm diag}(\sigma_{11}^{1/2}, \ldots, \sigma_{pp}^{1/2})$. Assume $\mu = 0$. We consider the following normalized version of (\ref{eq:1507172206}):
\begin{eqnarray}\label{eq:J021209}
 \sup_{u \ge 0} |\P( \sqrt n |D_0^{-1} \bar X_n|_\infty \ge u) - \P(|D_0^{-1} Z|_\infty \ge u)| \to 0,
\end{eqnarray}

\begin{assumption}
\label{assumption:1507172202}
There exists a constant $c > 0$ such that $\min_{1 \leq j \leq p} \sigma_{jj} \ge c$.
\end{assumption}

To state Theorem \ref{th:1507140259}, we need to define the following quantities:  $\Theta_{q, \alpha}=\Upsilon_{q, \alpha} \wedge (\||X_{\cdot}|_\infty\|_{q, \alpha} \log p)$, $L_1 = (n^{1/q-1/2} (\log p)^{1/2} \Theta_{q, \alpha})^{1/(\alpha - 1/2 + 1/q)}$, $L_2 = (\Psi_{2, \alpha} \Psi_{2, 0} (\log p)^2 )^{1/\alpha}$, $W_1 = (\Psi_{3, 0}^6 + \Psi_{4, 0}^4) (\log (p n) )^7$, $W_2 = \Psi_{2, \alpha}^2 (\log (p n) )^4$, $W_3 = (n^{-\alpha} (\log (pn))^{3/2}  \Theta_{q, \alpha})^{1/(1/2-\alpha - 1/q)}$, $N_1 = (n / \log p)^{q/2} / \Theta_{q, \alpha}^q$, $N_2 = n (\log p)^{-2} \Psi_{2, \alpha}^{-2}$, $N_3 = (n^{1/2} (\log p)^{-1/2}  \Theta_{q, \alpha}^{-1})^{1/(1/2-\alpha)}$.

\begin{thm}
\label{th:1507140259}
Let Assumption \ref{assumption:1507172202} be satisfied. (i) Assume that $\Theta_{q, \alpha} < \infty$ holds with some $q \geq 4$ and $\alpha > 1/2 - 1/q$ (the weaker dependence case),
\begin{eqnarray}
\label{eq:1507142040}
\Theta_{q, \alpha} n^{1/q-1/2} (\log (p n) )^{3/2} \to 0
\end{eqnarray}
and
\begin{eqnarray}
\label{eq:1507142041}
\max(L_1, L_2) \max(W_1, W_2) = o(1) \min(N_1, N_2).
\end{eqnarray}
Then the Gaussian Approximation (\ref{eq:J021209}) holds. (ii) Assume $0 < \alpha < 1/2 - 1/q$ (the stronger dependence case). Then (\ref{eq:J021209}) holds if $\Theta_{q, \alpha} (\log p)^{1/2} = o(n^{\alpha})$ and
\begin{eqnarray}
\label{eq:1508010523}
L_2 \max(W_1, W_2, W_3) = o(1) \min(N_2, N_3).
\end{eqnarray}
\end{thm}

\begin{rmk} (Optimality of our result on the allowed dimension $p$)
Assume $\alpha > 1/2 - 1/q$. In the special case with $\Psi_{q, \alpha} \asymp 1$ and $\Theta_{q, \alpha} \asymp p^{1/q}$, (\ref{eq:1507142040}) becomes
\begin{eqnarray}
\label{eq:1507142042}
p (\log (p n) )^{3q/2} = o(n^{q/2-1}),
\end{eqnarray}
which by elementary manipulations implies (\ref{eq:1507142041}), and hence the GA (\ref{eq:J021209}). It turns out that condition (\ref{eq:1507142042}), or equivalently $p (\log p)^{3q/2} = o(n^{q/2-1})$, is optimal up to a multiplicative logarithmic term. Consider the special case in which $X_{i j}$, $i, j \in \mathbb{Z},$ are i.i.d. symmetric random variables with $\E (X_{i j}^2) = 1$ and the tail probability $\P( X_{i j} \ge u) = u^{-p} \ell(u)$, $u \ge u_0$, where $\ell(u) = (\log u)^{-2}$. By \citet{nagaev1979large}, we have the expansion: for $y \ge \sqrt n$,
\begin{eqnarray}
\label{eq:J011029}
\P(X_{1 1} + \ldots + X_{n 1} \ge y) \sim n y^{-q} \ell(y) + 1 - \Phi(y/\sqrt n).
\end{eqnarray}
Let $M_n = X_{1 1} + \ldots + X_{n 1}$, $Z = (Z_1, \ldots, Z_p)^\top \sim N(0, \text{Id}_p)$ and assume
\begin{eqnarray}
\label{eq:1507142204}
n^{q/2-1} = o(p (\log n)^{-2} (\log p)^{-q/2}).
\end{eqnarray}
Then the Gaussian approximation (\ref{eq:J021209}) {\it does not hold}. To see this, let $u = (2 \log p)^{1/2}$. Then $p \P (|Z_1| \ge u) \to 0$, and, by (\ref{eq:J011029}) and (\ref{eq:1507142204}), $p \P( M_n \ge \sqrt n u) \to \infty$. Hence $\P^p( |M_n| \le \sqrt n u) \to 0$ and $\P^p(|Z_1| \le u) \to 1$, implying that
\begin{eqnarray*}
|\P( \sqrt n |\bar X_n|_\infty \le u) - \P(|Z|_\infty \le u)|
 &=&  |\P^p( |M_n| \le \sqrt n u) - \P^p(|Z_1| \le u)| \cr
 &=& | [1-2 \P( M_n \ge \sqrt n u)]^p - \P^p(|Z_1| \le u)|
 \to 1.
\end{eqnarray*}
Note that (\ref{eq:1507142204}) is equivalent to $n^{q/2-1} = o(p (\log p)^{-2-q/2})$, suggesting that (\ref{eq:1507142042}) is optimal up to a logarithmic term. \qed
\end{rmk}

Now suppose there exist $0 \le \kappa_1 \le \kappa_2$ such that $\Psi_{q, \alpha} \asymp p^{\kappa_1}$ and $\Theta_{q, \alpha} \asymp p^{\kappa_2}$, and $p^\tau \asymp n$. Elementary but tedious calculations show that, in the weaker dependence case $\alpha > 1/2 - 1/q$, if
\begin{eqnarray}
\label{eq:1507150901}
\tau > \max\left\{\frac{\kappa_2}{1/2-1/q}, \frac{2\kappa_1}{\alpha}+8\kappa_1, \frac{2}{q}\left(\frac{2\kappa_1}{\alpha}+8\kappa_1\right)+2\kappa_2\right\},
\end{eqnarray}
then conditions in (i) of Theorem \ref{th:1507140259} are satisfied, while for the stronger dependence case with $0 < \alpha < 1/2 - 1/q$, a larger sample size $n$ is required:
\begin{eqnarray}
\label{eq:1507150902}
\tau > \max\left\{\frac{\kappa_2}{\alpha}, \frac{2\kappa_1}{\alpha}+8\kappa_1,(1-2\alpha)\left(\frac{2\kappa_1}{\alpha}+8\kappa_1\right)+2\kappa_2 \right\}
\end{eqnarray}
The lower bounds in (\ref{eq:1507150901}) and (\ref{eq:1507150902}) are both non-decreasing of $\kappa_1, \kappa_2$ and non-increasing in $q, \alpha$.

Under (\ref{eq:1507142042}), the allowed dimension $p$ can only be at most a polynomial of $n$. To ensure the validity of GA in the ultra-high dimensional case with $\log p = o(n^c)$ with some $c > 0$, we need to consider the sub-exponential case in which $X_{i j}$ has finite moment with any order. For $\nu \ge 0$ and $\alpha \ge 0$, define the dependence-adjusted sub-exponential norm
\begin{eqnarray*}
\| X_{\cdot j} \|_{\psi_\nu, \alpha} = \sup_{q\ge 2}  { {\| X_{\cdot j} \|_{q, \alpha} } \over q^\nu} \mbox{ and }
 \Phi_{\psi_\nu, \alpha} = \max_{j \le p}  \| X_{\cdot j} \|_{\psi_\nu, \alpha}
\end{eqnarray*}
Let $L_3 = ( (\log p)^{1/\beta + 1/2}  \Phi_{\psi_\nu, \alpha} )^{1/\alpha}$, $N_4 = n (\log p)^{-1-2/\beta} \Phi_{\psi_\nu, 0}^{-2}$ and $W_4 = (\log (p n))^{3+2/\beta} \Phi_{\psi_\nu, 0}^2 + (\log (p n))^{4}$. Here  $\beta=2/(1+2\nu)$.

\begin{thm}
\label{th:1507140325}
Let Assumption \ref{assumption:1507172202} be satisfied. Assume that $\Phi_{\psi_\nu, \alpha}<\infty$ for some $\nu \geq 0$, $\alpha>0$ and
\begin{eqnarray}
\max(L_2, L_3) \max(W_1, W_4) = o(N_4),\quad  L_2^\alpha\max(W_1, W_4)=o(n).
\end{eqnarray}
Then the Gaussian Approximation (\ref{eq:J021209}) holds.
\end{thm}

\begin{proof}
The proof is similar to that of Theorem \ref{th:1507140259}, and thus is omitted.
\end{proof}

If $\Phi_{\psi_\nu, \alpha} \asymp 1$, then the ultra high-dimensional case with $\log p = o(n^c)$ with some $c>0$ is allowed, where specifically we can let
\begin{equation}
c=\left\{
\begin{array}{ll}
1/(8+2/\alpha+2/\beta), & 2/3 \leq \beta \leq 2\\
1/[7+(1/\beta+1/2)(1/\alpha+2)], & 1/2 \leq \beta <2/3 \\
1/[3+2/\beta+(1/\beta+1/2)(1/\alpha+2)], & 0 <\beta <1/2
\end{array}
\right..
\end{equation}

\subsection{Simultaneous Inference of Covariances}
Let $X_1, \ldots, X_n$ be i.i.d. $p$-dimensional vectors with mean $0$ and covariance matrix $\Gamma = \Gamma_0 = (\gamma_{j k})_{j, k=1}^p = \E (X_i X_i^\top)$. We can estimate $\Gamma$ by the sample covariance matrix $\hat \Gamma = (\hat \gamma_{jk})_{j, k=1}^p =  n^{-1} \sum_{i=1}^n X_i X_i^\top$.  To perform simultaneous inference on $\gamma_{j k}, 1\le j,k \le p$, one needs to derive asymptotic distribution of the maximum deviation $\max_{j, k \le p} | \hat \gamma_{jk} - \gamma_{jk}|$ or the normalized version $\max_{j, k \le p} | \hat \gamma_{jk} - \gamma_{jk}|/\tau_{j k}$; cf Equation (2) in \citet{Xiao20132899}. \citet{jiang2004} established the Gumbel convergence of the maximum deviation assuming that all entries of $X_i$ are also independent. See  \citet{Li2006} and \citet{liu2008} for some refined results. \citet{Xiao20132899} considered the extension which allows dependence among entries of $X_i$. However the latter paper requires that the vectors $X_1, \ldots, X_n$ are i.i.d. The problem of further extension to temporally dependent $X_i$ is open. In analyzing electrocorticogram data in the format of multivariate time series, \citet{PhysRevE.79.061916} proposed to use the maximum cross correlation between time series to identify edges that connect the corresponding nodes in a network, suggesting that an asymptotic theory for maximum deviations of sample covariances is needed.

Our Theorems \ref{th:1507140259} and \ref{th:1507140325} can be applied to the above problem of further extension to temporally dependent process $(X_i)$. Let $(X_i)$ be a mean zero $p$-dimensional stationary process of form (\ref{highdimensionrepresentation}).  To apply Theorems \ref{th:1507140259} and \ref{th:1507140325}, one needs to deal with the key issue of computing the functional dependence measure of the $p^2$-dimensional vector $\mathcal{X}_i = \mbox{vec}(X_i X_i^\top-\E(X_i X_i^\top))$. Interestingly, our framework allows a natural and elegant treatment. Let $a = (j, k)$, $j, k \le p$, and $\mathcal{X}_{i a} =  X_{ij}X_{ik} - \gamma_a$, where $ \gamma_a = \gamma_{j k} = \E (X_{ij}X_{ik})$. By H\"{o}lder's inequality, the functional dependence of the component process $(\mathcal{X}_{i a})_i$
\begin{eqnarray}\label{eq:1508220838}
\varphi_{i, q/2, a} &:= &\|X_{ij}X_{ik}-\E (X_{ij}X_{ik})-X_{ij, \{0\}}X_{ik, \{0\}}
 +\E(X_{ij, \{0\}}X_{ik, \{0\}})\|_{q/2} \cr
&\leq&2\|X_{ij}X_{ik}-X_{ij, \{0\}}X_{ik, \{0\}}\|_{q/2}\cr
&\leq&2\|X_{ij}(X_{ik}-X_{ik, \{0\}})\|_{q/2} +
 2\| (X_{ij} -X_{ij, \{0\}}  ) X_{ik, \{0\}}\|_{q/2}\cr
&\leq &2\|X_{ij}\|_q\delta_{i,q,k} + 2\|X_{ik}\|_q\delta_{i,q,j}.
\end{eqnarray}
Hence, we can have an upper bound of the dependence adjusted norm of $(\mathcal{X}_{i a})$
\begin{eqnarray}\label{eq:1508220839}
\|\mathcal{X}_{\cdot a}\|_{q/2,\alpha}
 &:=& \sup_{m \geq 0} (m+1)^\alpha \sum_{i=m}^\infty \varphi_{i,q/2,j,k}\cr
 &\leq &2\|X_{\cdot j}\|_{q,0}\|X_{\cdot k}\|_{q,\alpha}+2\|X_{\cdot k}\|_{q,0}\|X_{\cdot j}\|_{q,\alpha}.
\end{eqnarray}
Consequently, the uniform and the overall dependence adjusted norms of $\mathcal{X}_i$ are
\begin{eqnarray}\label{eq:1508220840}
&& \max_a \|\mathcal{X}_{\cdot a}\|_{q/2, \alpha} \leq 4 \Psi_{q,0}\Psi_{q,\alpha},\cr
&& \left(\sum_a \|\mathcal{X}_{\cdot a}\|_{q/2, \alpha}^{q/2}\right)^{2/q}
 \leq 4 \left(\sum_{j=1}^p\|X_{\cdot j}\|_{q,0}^{q/2}\right)^{2/q}
  \left(\sum_{j=1}^p\|X_{\cdot j}\|_{q,\alpha}^{q/2}\right)^{2/q}.
\end{eqnarray}
Similarly, the $\mathcal{L}^\infty$ dependence adjusted norm for the process $(\mathcal{X}_i)$ can be calculated by
\begin{equation}\label{eq:1508220841}
\||\mathcal{X}_{\cdot}|_\infty \|_{q/2, \alpha} \leq 4 \||X_{\cdot}|_\infty\|_{q, 0} \||X_{\cdot}|_\infty\|_{q, \alpha}.
\end{equation}
With (\ref{eq:1508220838})-(\ref{eq:1508220841}), conditions in Theorems \ref{th:1507140259} and \ref{th:1507140325} can be formulated accordingly, and under those conditions we can have the following Gaussian Approximation
\begin{eqnarray}\label{eq:1508220832}
 \sup_{u \ge 0} |\P( \sqrt n \max_a | \hat \gamma_a - \gamma_a| / \tau_a \ge u)
  - \P(\max_a |Z_a / \tau_a| \ge u)| \to 0,
\end{eqnarray}
where $Z = (Z_a)_a \sim N(0, \Sigma_\mathcal{X})$, $\Sigma_\mathcal{X}$ is the $p^2 \times p^2$ long-run covariance matrix of $(\mathcal{X}_i)_i$ and $(\tau_a^2)_a$ is the diagonal matrix of $\Sigma_\mathcal{X}$.

\section{Estimation of long-run covariance matrix}
\label{sec:estimation of long-run variance}

Given the realization $X_1, \ldots, X_n$, to apply the Gaussian approximation (\ref{eq:J021209}), we need to estimate the long-run covariance matrix $\Sigma$. Note that $\Sigma/(2\pi)$ is the value of the spectral density matrix of $(X_i)$ at zero frequency. In the one-dimensional case, there is a large literature concerning spectral density estimation; see for example \citet{anderson1971}, \citet{priestley1981}, \citet{rosenblatt1985}, \citet{brockwell1991}, \citet{liu2010asymptotics} among others. In the high-dimensional setting, \citet{chen2013} studied the regularized estimation of $\Gamma(0) = \E (X_0 X_0^\top)$.  Assume $\E X_i = 0$. We then consider the batched mean estimate
\begin{equation}
\label{estimate LRCM}
\hat{\Sigma}=\frac{1}{M w}\sum_{b=1}^w Y_bY_b^\top=\frac{1}{M w}\sum_{b=1}^w (\sum_{i \in L_b}X_i)(\sum_{i \in L_b}X_i)^\top.
\end{equation}
where the window $L_b = \{1+(b-1)M, \ldots, nM\}, b=1, \ldots, w$, the window size $|L_b| = M \to \infty$ and the number of blocks $w = \lfloor n/M \rfloor$. Theorems \ref{theoremforquadraticformunderpolynomial} and \ref{theoremforquadraticformunderexponential} concern the convergence of the above estimate for processes with finite polynomial and finite sub-exponentail dependence adjusted norms, respectively. The convergence rate depends in a subtle way on the temporal dependence characterized by $\alpha$ (cf. (\ref{eq:2015081932})), the uniform and the overall dependence adjusted norms $\Psi_{q, \alpha}$ and $\Upsilon_{q, \alpha}$, respectively, the same size $n$ and the dimension $p$.

For a random variable $X$, we define the operator $\E_0$ as $\E_0 (X) := X-\E X$.

\begin{thm}
\label{theoremforquadraticformunderpolynomial}
Assume $\Psi_{q, \alpha} < \infty$ with $q>4$ and $\alpha>0$, and $M=O(n^{\varsigma})$ for some $0< \varsigma <1$. Let $F_\alpha = w M$ (resp. $w M^{q/2-\alpha q/2}$ or $w^{q/4-\alpha q/2} M^{q/2-\alpha q/2}$) for $\alpha>1-2/q$ (resp. $1/2-2/q< \alpha <1-2/q$ or $\alpha < 1/2-2/q$). Then for $x \geq \sqrt{w}M \Psi_{q, \alpha}^2$, we have
 \begin{eqnarray}\label{eq:201508181936}
&&\P(n|\diag(\hat{\Sigma})-\E \diag(\hat{\Sigma})|_\infty\geq x)
\lesssim {{ F_\alpha \Upsilon_{q, \alpha}^q } \over {x^{q/2}}}+ p \exp\left(-\frac{C_{q, \alpha} x^2}{wM^2\Psi_{4, \alpha}^4}\right),
\cr
&& \P(n|\hat{\Sigma}-\E \hat{\Sigma}|_\infty \geq x)
\lesssim {{p F_\alpha \Upsilon_{q, \alpha}^q } \over {x^{q/2}}}
 + p^2 \exp\left(-\frac{C_{q, \alpha} x^2}{wM^2\Psi_{4, \alpha}^4}\right)
\end{eqnarray}
for all large $n$, where the constants in $\lesssim$ only depend on $\varsigma$, $\alpha$ and $q$.
\end{thm}

\begin{proof}
Fix $1 \leq j, k \leq p$; let $T=\sum_{b=1}^w Y_{bj} Y_{bk}$, where $Y_{b j} = \sum_{i \in L_b}X_{i j}$.
For $\tau \geq 0$, define $X_{ij, \tau}=\E(X_{ij}|\eps_{i-\tau}, \ldots, \eps_{i})$, $Y_{bj, \tau}=\sum_{i \in L_b}X_{ij, \tau}$ and $T_{\tau}=\sum_{b=1}^w Y_{bj,\tau}Y_{bk, \tau}$.
We will first prove for any $x>0$
\begin{equation}
\P(|\E_0(T-T_{M})| \geq x)
\lesssim \left\{
\begin{array}{ll}
\label{differenceforquadraticform}
x^{-q/2}wM^{q/2-\alpha q/2}\xi_{q, \alpha}^{q/2}+E_{q, \alpha}(x), & \alpha >1/2-2/q \\
x^{-q/2}w^{q/4-\alpha q/2}M^{q/2-\alpha q/2}\xi_{q, \alpha}^{q/2}+E_{q, \alpha}(x), & \alpha <1/2-2/q \\
\end{array}
\right. ,
\end{equation}
where the constants in $\lesssim$ only depend on $\varsigma$, $\alpha$ and $q$, and
\begin{eqnarray*}
&& \xi_{q, \alpha}=\|X_{\cdot j}\|_{q, 0}\|X_{\cdot k}\|_{q, \alpha}+\|X_{\cdot k}\|_{q, 0}\|X_{\cdot j}\|_{q, \alpha},\\
&& E_{q, \alpha}(x)=\exp\{-C_{q, \alpha}(wM^{2-2\alpha}\xi_{4, \alpha}^2)^{-1}x^{2}\}.
\end{eqnarray*}

Following the argument in the proof of Lemma \ref{lemma for m-approximation, polynomial tail}, let $L=\lfloor (\log w)/ (\log 2)\rfloor$, $\varpi_l=2^l$, $1 \leq l <L$, $\varpi_L=w$ and $\tau_l= M \varpi_l$ for $1 \leq l \leq L$. Let $\varpi_0 = 1$ and $\tau_0 = M$. Write
\begin{equation}
\label{T-T_{M}}
T-T_{M} = T-T_{Mw} + \sum_{l=1}^L V_{w, l}, \mbox{ where } V_{w, l}=T_{\tau_l}-T_{\tau_{l-1}}.
\end{equation}
By the argument in Lemma 9 of \citet{xiao2012covariance}, we have
\begin{eqnarray}
\| \E_0(T-T_{Mw})\|_{q/2}
&\leq & C_q M\sqrt{w}(\Delta_{0, q, j}\Delta_{Mw+1, q, k}+\Delta_{Mw+1, q, j}\Delta_{0, q, k})\nonumber \\ \label{appximationmoment}
&\leq & C_q M\sqrt{w}(Mw)^{-\alpha}\xi_{q, \alpha}
\end{eqnarray}
for some constant $C_q > 0$. By Markov's inequality, for $x > 0$,
\begin{equation}
\label{T-T_{Mw}}
\P(|\E_0(T-T_{Mw})|\geq x)\leq \frac{C_q M^{q/2-\alpha q/2}w^{q/4-\alpha q/2}\xi_{q, \alpha}^{q/2}}{x^{q/2}}.
\end{equation}
By the same argument for proving (\ref{appximationmoment}), we have
\begin{equation*}
\|\E_0(V_{w, l})\|_{q/2} \leq C_q M\sqrt{w} \tau_l^{-\alpha}\xi_{q, \alpha}.
\end{equation*}
Let $c=q/4-1-\alpha q/2$, $\lambda_l= 3 l^{-2} \pi^{-2}$ if $1 \leq l \leq L/2$ and $\lambda_l= 3 (L+1-l)^{-2} \pi^{-2}$ if $L/2 < l \leq L$. Then $\sum_{l=1}^L \lambda_l < 1$.  By the Nagaev (1979) inequality, it follows that
\begin{eqnarray}
\P(|\sum_{l=1}^L \E_0(V_{w, l})| \geq x) &\leq& \sum_{l=1}^L \P(|\E_0(V_{w, l})| \geq \lambda_lx) \nonumber \\
& \leq & \sum_{l=1}^L \frac{C_1 w \varpi_l^{-1}(M{\varpi_l}^{1/2} \tau_l^{-\alpha})^{q/2}\xi_{q, \alpha}^{q/2}}{(\lambda_l x)^{q/2}}+4 \sum_{l=1}^L \exp \left(-\frac{C_2 (\lambda_l x)^2\tau_l^{2 \alpha}}{wM^2\xi_{4, \alpha}^{2}} \right)\nonumber \\
&\leq & \frac{C_3wM^{q/2-\alpha q/2}\xi_{q, \alpha}^{q/2}}{x^{q/2}}\sum_{l=1}^L \frac{\varpi_l^{c}}{\lambda_l^{q/2}}+C_4 \sum_{l=1}^L E_{q, \alpha}(\lambda_l \varpi_l^\alpha x). \label{V_{w,l}}
\end{eqnarray}
Elementary calculations show that
\begin{equation}
\label{poly}
\sum_{l=1}^L \frac{\varpi_l^{c}}{\lambda_l^{q/2}} \leq C_5 \text{ for } c<0 \text{ and } \sum_{l=1}^L \frac{\varpi_l^{c}}{\lambda_l^{q/2}} \leq C_6 \varpi_L^c=C_6 w^c \text{ for } c>0.
\end{equation}
Furthermore, we can use (\ref{exponentialterm}) to obtain
\begin{equation}
\label{exp}
\sum_{l=1}^L E_{q, \alpha}(\lambda_l \varpi_l^\alpha x) \lesssim E_{q, \alpha} (x).
\end{equation}
Putting (\ref{T-T_{M}}), (\ref{T-T_{Mw}}), (\ref{V_{w,l}}), (\ref{poly}) and (\ref{exp}) together, we then have (\ref{differenceforquadraticform}).

Now it suffices to consider $\P(|\E_0(T_M)|\geq x)$. Observe that $(Y_{bj, M}Y_{bk, M})_{b \text{ is odd}}$ are independent and so are $(Y_{bj, M}Y_{bk, M})_{b \text{ is even}}$. By Corollary 1.7 of \citet{nagaev1979large}, for any $J > 1$,
\begin{eqnarray*}
\P(|\E_0(T_M)|\geq x) &\leq & \sum_{b=1}^w \P(|\E_0(Y_{bj, M}Y_{bk, M}
)|\geq x/(2J))+2\left(\frac{\sum_{b=1}^w \|\E_0(Y_{bj, M}Y_{bk, M})\|_{q/2}^{q/2}}{J x^{q/2}}\right)^{J}\\
&&+4\exp\left\{-\frac{C_q x^2}{\sum_{b=1}^w \|\E_0(Y_{bj, M}Y_{bk, M})\|_2^2}\right\}.
\end{eqnarray*}
Note that $\|Y_{bj, M}\|_{q} \leq C_q \sqrt M \|X_{\cdot j}\|_{q, 0}$. Hence for $1 \leq b \leq w$, $1 \leq j, k\leq p$ and $q \geq 4$,
\begin{equation*}
\|\E_0(Y_{bj, M}Y_{bk, M})\|_{q/2} \leq 2\|Y_{bj, M}Y_{bk, M}\|_{q/2} \leq 2\|Y_{bj, M}\|_{q}\|Y_{bk, M}\|_{q}\leq C_q M \|X_{\cdot j}\|_{q, 0}\|X_{\cdot k}\|_{q, 0}.
\end{equation*}
Since
\begin{equation*}
\E |Y_{bj, M}Y_{bk, M}| \leq \|Y_{bj, M}\|_2 \|Y_{bk, M}\|_2  \leq M \|X_{\cdot j}\|_{2, 0}\|X_{\cdot k}\|_{2, 0}
 \le {x \over \sqrt w},
\end{equation*}
we have
\begin{eqnarray*}
\P(|\E_0(T_M)|\geq x)& \leq& \sum_{b=1}^w \P(|Y_{bj, M}Y_{bk, M}|\geq x/(4J)) \\
&&+ 2\left(\frac{wM^{q/2}\|X_{\cdot j}\|_{q,0}^{q/2}\|X_{\cdot k}\|_{q,0}^{q/2}}{J x^{q/2}}\right)^{J}
 +4\exp\left(-\frac{C_q x^2}{wM^2\Psi_{4,0}^4}\right).
\end{eqnarray*}
Recall that $M=O(n^\varsigma)$ with $0 < \varsigma < 1$. Let $J = 1+ (2q-2) (q-4)^{-1} (1- \varsigma)^{-1}$. Since  $x \geq \sqrt{w}M \|X_{\cdot j}\|_{q,0}\|X_{\cdot k}\|_{q,0}$, elementary calculations show that for sufficiently large $n$ the second term in the above expression is no greater than $C_J wM\|X_{\cdot j}\|_{q,0}^{q/2}\|X_{\cdot k}\|_{q,0}^{q/2}/x^{q/2}$. As for the first term, we have
\begin{equation*}
\P(|Y_{bj, M}Y_{bk, M}| \geq x/(4J))  \leq  \P(|Y_{bj, M}| \geq \sqrt{x/(4J)})+\P(|Y_{bk, M}| \geq \sqrt{x/(4J)}).
\end{equation*}
By Lemma \ref{theorem for the sum polynomial}, for $\alpha>1/2-1/q$ and $\alpha <1/2-1/q$, respectively, we have
\begin{equation*}
\P(|Y_{bj, M}| \geq \sqrt{x}) \leq \left\{
\begin{array}{ll}
C_{q, \alpha}x^{-q/2}M\|X_{\cdot j}\|_{q, \alpha}^q+ C_{q, \alpha} \exp\left(-\frac{C_{q, \alpha}x}{M \|X_{\cdot j}\|_{2, \alpha}^2} \right), \\
C_{q, \alpha}x^{-q/2}M^{q/2-\alpha q}\|X_{\cdot j}\|_{q, \alpha}^q+ C_{q, \alpha} \exp\left(-\frac{C_{q, \alpha}x}{M \|X_{\cdot j}\|_{2, \alpha}^2} \right). \\
\end{array}
\right.
\end{equation*}
A similar inequality holds for $\P(|Y_{bk, M}| \geq \sqrt{x})$. Let $\phi_{q, \alpha}=\|X_{\cdot j}\|_{q, \alpha}^q+\|X_{\cdot k}\|_{q, \alpha}^q$. Hence, it follows that  for $\alpha>1/2-1/q$ and $\alpha <1/2-1/q$ respectively,
\begin{equation}
\P(|\E_0(T_M)|\geq x) \leq \left\{
\begin{array}{ll}
\label{maintermforquadraticform}
C_{q, \alpha}x^{-q/2} wM \phi_{q, \alpha} + C_{q, \alpha} \exp\left(-\frac{C_{q, \alpha}x^2}{wM^2 \Psi_{4, \alpha}^4} \right), \\
C_{q, \alpha}x^{-q/2}w M^{q/2-\alpha q}\phi_{q, \alpha}+ C_{q, \alpha} \exp\left(-\frac{C_{q, \alpha}x^2}{wM^2 \Psi_{4, \alpha}^4} \right). \\
\end{array}
\right.
\end{equation}
Combining (\ref{differenceforquadraticform}) and (\ref{maintermforquadraticform}), and noticing that $\xi_{q, \alpha}^{q/2} \leq C_q\phi_{q, \alpha}$, it follows that
\begin{eqnarray*}
 \P(|\E_0(T)|\geq x) \le C_{q, \alpha} x^{-q/2} F_\alpha \phi_{q, \alpha}
  +C_{q, \alpha} \exp\left(-\frac{C_{q, \alpha} x^2}{wM^2\Psi_{4, \alpha}^4}\right).
\end{eqnarray*}
which implies (\ref{eq:201508181936}) by the Bonferroni inequality by summing over $j$ and $k$.
\end{proof}

Under stronger moment conditions, we can have an exponential inequality.

\begin{thm}
\label{theoremforquadraticformunderexponential}
Assume $\Phi_{\psi_\nu, 0} < \infty$ for some $\nu \geq 0$. Then for all $x >0$, we have
\begin{eqnarray}
&& \P(n|\diag(\hat{\Sigma})-\E \diag(\hat{\Sigma})|_\infty \geq x) \lesssim   p \exp\left(-\frac{x^\gamma}{4e\gamma (\sqrt{w}M\Phi_{\psi_\nu, 0}^2)^\gamma }\right), \label{exponentialforquadraticform1}\\
&&\P(n|\hat{\Sigma}-\E \hat{\Sigma}|_\infty \geq x) \lesssim  p^2 \exp\left(-\frac{x^\gamma}{4e\gamma (\sqrt{w}M\Phi_{\psi_\nu, 0}^2)^\gamma }\right), \label{exponentialforquadraticform2}
\end{eqnarray}
where $\gamma=1/(1+2 \nu)$ and the constants in $\lesssim$ only depend on $\nu$.
\end{thm}
\begin{proof}
Let $T=\sum_{b=1}^w Y_{bj} Y_{bk}$. By the Burkholder inequality, we have
\begin{eqnarray}\label{eq:1508182136}
\| \E_0 T \|_{q/2}^2 \le (q/2-1) \sum_{l=-\infty}^{w M} \|{\cal P}^l T \|_{q/2}^2
 \le (q/2-1) \sum_{l=-\infty}^{w M} \left( \sum_{b=1}^w  \|{\cal P}^l Y_{bj} Y_{bk} \|_{q/2} \right)^2
\end{eqnarray}
By Theorem 3 in \citet{wu2011asymptotic}, $\| Y_{bj}  \|_q \le (q-1)^{1/2} \sqrt M \|X_{\cdot j}\|_{q, 0}$. Since $\| Y_{bk} - Y_{bk, \{l\}} \|_{q}  \le \sum_{h=1+(b-1)M}^{b M} \delta_{h-l, q, k}$, we have
\begin{eqnarray*}
\sum_{b=1}^w  \|{\cal P}^l Y_{bj} Y_{bk} \|_{q/2}
 &\le& \sum_{b=1}^w   \| Y_{bj} Y_{bk} - Y_{bj, \{l\} } Y_{bk, \{l\}} \|_{q/2}  \cr
  &\le& \sum_{b=1}^w  ( \| Y_{bj} \|_q \| Y_{bk} - Y_{bk, \{l\}} \|_{q}  +
 \| Y_{bj} - Y_{bj, \{l\} } \|_q  \|Y_{bk, \{l\}}\|_q ) \cr
  &\le&  (q-1)^{1/2} \sqrt M \left(\|X_{\cdot j}\|_{q, 0} \sum_{h=1}^{w M} \delta_{h-l, q, k}
  + \|X_{\cdot k}\|_{q, 0} \sum_{h=1}^{w M} \delta_{h-l, q, j} \right),
\end{eqnarray*}
which by (\ref{eq:1508182136}) implies that
\begin{eqnarray}\label{eq:1508182144}
\| \E_0 T \|_{q/2}^2 \le (q/2-1) \sum_{l=-\infty}^{w M} \|{\cal P}^l T \|_{q/2}^2
 \le (q-2)(q-1) w M^2 \|X_{\cdot j}\|_{q, 0}^2 \|X_{\cdot k}\|_{q,0}^2.
 \end{eqnarray}
Let $R_{j k} = \E_0 T /(\sqrt w M)$. Similarly as the argument for proving Lemma \ref{m-approximationtheorem2},
if $\gamma h \geq 2$, it follows that $\|R_{jk}\|_{\gamma h} \leq (2\gamma h-1)(2\gamma h)^{2 \nu}\|X_{\cdot j}\|_{\psi_\nu, 0}\|X_{\cdot k}\|_{\psi_\nu, 0}$. Let $\tau_0=(2e\gamma \|X_{\cdot j}\|_{\psi_\nu, 0}^\gamma\|X_{\cdot k}\|_{\psi_\nu, 0}^\gamma)^{-1}$. Notice that $-2\nu=1-1/\gamma$. Then
\begin{eqnarray*}
\frac{t^h\|R_{jk}^\gamma\|^h_h}{h!} &\leq& \frac{t^h(2\gamma h-1)^{\gamma h}(2\gamma h)^{2 \nu \gamma h}\|X_{\cdot j}\|_{\psi_\nu, 0}^{\gamma h}\|X_{\cdot k}\|_{\psi_\nu, 0}^{\gamma h}}{C_1(h/e)^h a_h^{-1}}\\
&\leq & \frac{a_h t^h (2\gamma h -1)^{\gamma h}}{C_1 \tau_0^h(2\gamma h)^{\gamma h}} \leq \frac{a_h t^h}{C_1 \sqrt{e} \tau_0^h}.
\end{eqnarray*}
If $\gamma h <2$, then $\|R_{jk}\|_{\gamma h} \leq \|R_{jk}\|_2 \leq \sqrt{6}\cdot 4^{2 \nu} \|X_{\cdot j}\|_{\psi_\nu, 0}\|X_{\cdot k}\|_{\psi_\nu, 0}$. So we have
\begin{eqnarray*}
\E [\exp(t R_{jk}^\gamma)]&\leq & 1+ \sum_{1 \leq h <2/\gamma}\frac{t^h(\sqrt{6}\cdot 4^{2 \nu} \|X_{\cdot j}\|_{\psi_\nu, 0}\|X_{\cdot k}\|_{\psi_\nu, 0})^{\gamma h}}{h!}+\sum_{h \geq 2/\gamma}\frac{a_h t^h}{C_1 \sqrt{e} \tau_0^h}\\
&\leq & 1+C_\gamma \sum_{h=1}^\infty a_h\frac{t^h}{\tau_0^h} \leq 1+ C_\gamma \frac{t/\tau_0}{(1-t/\tau_0)^{1/2}}.
\end{eqnarray*}
By choosing $t=\tau_0/2$, and applying the Markov inequality and the Bonferroni inequality, (\ref{exponentialforquadraticform1}) and (\ref{exponentialforquadraticform2}) are obtained.
\end{proof}

\begin{rmk}
An alternative estimate of $\Sigma$, which also works with unknown mean $\E X_i$, is
\begin{equation}
\label{long run covariance estimate minus mean}
\tilde{\Sigma}=\frac{1}{w M}\sum_{b=1}^w (\sum_{i \in L_b}X_i-M \bar{X})(\sum_{i \in L_b}X_i-M\bar{X})^\top,
\end{equation}
where $\bar{X} = (w M)^{-1}\sum_{i=1}^{w M} X_i$, $w = \lfloor n/M \rfloor$. Then $|\tilde{\Sigma} - \hat{\Sigma}|_\infty = M |\bar X|_\infty^2$. Applying Lemma \ref{theorem for the sum polynomial} to $\sum_{i=1}^{w M} X_{i j}$, one can conclude that Theorems \ref{theoremforquadraticformunderpolynomial} and \ref{theoremforquadraticformunderexponential} still hold for $\tilde{\Sigma}$ with $\E\hat{\Sigma}$ therein replaced by $\Sigma_M:=\sum_{i=-M}^M (1-|i|/M) \Gamma_i$ (which equals to $\E\hat{\Sigma}$ if $\E X_i = 0$).
\end{rmk}

\begin{cor}
\label{corollary for Sigma}
(i) Under conditions in Theorem \ref{theoremforquadraticformunderpolynomial},  we have $| \tilde{\Sigma} - \Sigma|_\infty = O_\P(r_n)$, where
 \begin{equation*}
 r_n =
n^{-1} \max\{p^{2/q} F^{2/q}_\alpha \Upsilon_{q, \alpha}^2, \, \sqrt w M \Psi_{4, \alpha}^2 \sqrt{\log p}, \,
 \sqrt w M \Psi_{q, \alpha}^2 \}  + \Psi_{2,0}\Psi_{2, \alpha} v(M),
 \end{equation*}
where $v(M)=1/M$ if $\alpha>1$, $v(M)=\log M/M$ if $\alpha=1$ and $v(M)=1/M^\alpha$ if $0<\alpha <1$. (ii) Under conditions in Theorem \ref{theoremforquadraticformunderexponential}, we have $| \tilde{\Sigma} - \Sigma|_\infty = O_\P(r_n)$ with $r_n =  n^{-1} \sqrt{w}M \Phi_{\psi_\nu,0}^2 (\log p)^{1/\gamma} + \Psi_{2,0}\Psi_{2, \alpha} v(M)$.
\end{cor}

The above Corollary easily follows from Theorems \ref{theoremforquadraticformunderpolynomial} and \ref{theoremforquadraticformunderexponential} since the bias $|\Sigma_M-\Sigma|_{\infty} \lesssim \Psi_{2,0}\Psi_{2, \alpha} v(M)$; see the proof of Lemma \ref{lemmafortwogaussian}.

For the estimate $\tilde{\Sigma}$ in (\ref{long run covariance estimate minus mean}), let $\tilde{D}_0=[\diag(\tilde{\Sigma})]^{1/2}$. Let $\tilde{Z} = \tilde{\Sigma}^{1/2} \eta$, where $\eta \sim N(0, \text{Id}_p)$ is independent of $(X_i)_i$. Then conditioning on $(X_i)_i$, $\tilde{Z} \sim N(0, \tilde{\Sigma})$. Let $0 < \theta < 1$; let $\tilde \chi_\theta$ be the conditional $\theta$-quantile of $|\tilde{D}_0^{-1}\tilde{Z}|_\infty$ given $(X_i)_{i=1}^n$. We can use $\tilde \chi_\theta$ to estimate the $\theta$-quantile of $|D_0^{-1}(\bar{X}_n-\mu)|_\infty$, thus constructing simultaneous confidence intervals for $\mu = (\mu_1, \ldots, \mu_p)^\top$ as $\hat \mu_j \pm \tilde \chi_\theta \tilde \sigma_{j j}^{1/2}$, $1 \le j \le p$. Assume that $r_n = o(1/\log^2 p)$. Then $\pi(|\tilde{\Sigma}-\Sigma|_\infty)=o(1)$, and by Lemma 3.1 in \citet{chernozhukov2013}, the latter simultaneous confidence intervals have the asymptotically correct coverage probability $\theta$. Note that $\tilde \chi_\theta$ can be obtained by sample quantile estimates from extensive simulations of $\tilde{Z} = \tilde{\Sigma}^{1/2} \eta$.

\section{Tail probability inequalities under dependence}
\label{sec:probability inequalities}
Tail probability inequalities play an important role in simultaneous inference. Here we shall provide some Nagaev-type tail probability inequalities. They are of independent interest. Let $\eps_i, \eps'_j, i, j, \in \mathbb{Z}$, be i.i.d. random variables. We start with the one-dimensional stationary process $(e_i)_{i=-\infty}^\infty$ of the form
\begin{equation}
\label{representation}
e_i=g(\ldots, \eps_{i-1}, \eps_i),
\end{equation}
where $g$ is a measurable function such that $e_i$ is well-defined. Recall $\calF_i^j=(\eps_i, \eps_{i+1}, \ldots, \eps_j)$, $\calF^j=(\ldots, \eps_{j-1}, \eps_j)$ and $\calF_{i}=(\eps_i, \eps_{i+1}, \ldots)$. Let the projection operators $\mathcal{P}^0 \cdot = \E( \cdot |\calF^{0}) - \E( \cdot | \calF^{-1})$, $\mathcal{P}_0 \cdot = \E( \cdot |\calF_{0}^i) - \E( \cdot |\calF_{1}^i)$. As in (\ref{eq:1507140845}), define respectively the functional and the predictive dependence measures
\begin{eqnarray*}
 \delta_{i, q}=\| e_i-g(\calF^{i, \{0\}}) \|_q, \,\, \theta_{i, q} =\|\mathcal{P}^0 e_i\|_q, \,\, \mbox{ and }
 \theta_{i, q}' =\|\mathcal{P}_0 e_i\|_{q},
\end{eqnarray*}
where $\calF^{i, \{0\}} = ( \ldots, \eps_{-1}, \eps_0', \eps_{1}, \ldots, \eps_i)$. Let $\delta_{i, q}= 0$ if $i < 0$; let $\Delta_{m, q} = \sum_{i=m}^\infty \delta_{i, q}$, $m \ge 0$, be the tail dependence measures, and the dependence adjusted norm
\begin{equation}
\label{dependenceadjustednorm}
\| e_\cdot \|_{q, \alpha}:= \sup_{m \geq 0} (m+1)^{\alpha} \Delta_{m, q}, \text{ for }\alpha \geq 0.
\end{equation}
Here $\delta_{i, q}$ measures the dependence of $e_i$ on $\eps_0$ and $\Delta_{m, q}$ measures the cumulative impact of $\eps_0$ on $(e_i)_{i \geq m}$. The projections $(\mathcal{P}_{-i}\cdot)_{i \in \mathbb{Z}}$ and $(\mathcal{P}^{i}\cdot)_{i \in \mathbb{Z}}$ induces martingale differences with respect to $(\calF_{-i})$ and $(\calF^i)$, respectively. Both predictive dependence measures provide an evaluation to the effect on the prediction of $e_i$ when part of the previous inputs is concealed, and they satisfy $\theta_{i, q} \leq \delta_{i, q}$ and $\theta'_{i, q} \leq \delta_{i, q}$ in view of Jensen's inequality.

\subsection{Inequalities with Finite Polynomial Moments}
For $m\geq 0$, the $m$-dependence approximation of $e_i$ is denoted by $e_{i, m}$ where
\begin{equation*}
e_{i, m}=\E(e_i|\eps_{i-m},\eps_{i-m+1}, \ldots, \eps_i).
\end{equation*}
Let $S_n=\sum_{i=1}^n e_i$, $S_{n, m}=\sum_{i=1}^n e_{i, m}$.
With the dependence adjusted norm (\ref{dependenceadjustednorm}), we are able to provide tail probability inequalities for error bounds when approximating $(e_i)$ by the $m$-dependent process $(e_{i, m})$. In lemmas below the constant $C_{q, \alpha}$ only depends on $q$ and $\alpha$ and its values may change from line to line.

\begin{lem}
\label{m-approximationtheorem1}
Assume $\|e_\cdot\|_{q, \alpha} < \infty$, where $q >2$ and $\alpha>0$. (i) If $\alpha>1/2-1/q$, then
\begin{equation}
\label{alpha>}
\P(|S_n-S_{n, m}| \geq x) \leq \frac{C_{q, \alpha} n m^{q/2-1-\alpha q}\|e_\cdot\|_{q, \alpha}^q}{x^q}+ C_{q,\alpha} \exp\lp - \frac{C_{q, \alpha} x^2 m^{2 \alpha}}{n\|e_\cdot\|_{2, \alpha}^2}\rp
\end{equation}
holds for all $x>0$ and $1 \leq m \leq n$.
(ii) If $0 < \alpha < 1/2-1/q$, we have
\begin{equation}
\label{alpha<}
\P(|S_n-S_{n, m}| \geq x) \leq \frac{C_{q,\alpha} n^{q/2-\alpha q} \|e_\cdot\|_{q, \alpha}^q}{x^q}+ C_{q, \alpha} \exp\lp - \frac{C_{q, \alpha} x^2 m^{2 \alpha}}{n\|e_\cdot\|_{2, \alpha}^2}\rp.
\end{equation}
\end{lem}

\begin{proof}[Proof of Lemma \ref{m-approximationtheorem1}]
It is a special case of Lemma \ref{lemma for m-approximation, polynomial tail} for $p=1$.
\end{proof}

\begin{lem}[cf. Theorem 2 of \citet{wu2014}]
\label{theorem for the sum polynomial}
Assume that $\|e_\cdot\|_{q, \alpha} < \infty$, where $q >2$ and $\alpha>0$. (i) If $\alpha>1/2-1/q$, then there exists some constant $C_{q, \alpha}$ depending on $q$ and $\alpha$ only such that,  for $x>0$,
\begin{equation}
\label{eq:1508220935}
\P(|S_n| \geq x) \leq \frac{C_{q, \alpha} n \|e_\cdot\|_{q, \alpha}^q}{x^q}+ C_{q,\alpha} \exp\lp - \frac{C_{q, \alpha} x^2}{n\|e_\cdot\|_{2, \alpha}^2}\rp.
\end{equation}
(ii) If $0 < \alpha < 1/2-1/q$, we have the following inequality,
\begin{equation}
\label{eq:1508220936}
\P(|S_n| \geq x) \leq \frac{C_{q,\alpha} n^{q/2-\alpha q} \|e_\cdot\|_{q, \alpha}^q}{x^q}+ C_{q, \alpha} \exp\lp - \frac{C_{q, \alpha} x^2}{n\|e_\cdot\|_{2, \alpha}^2}\rp.
\end{equation}
\end{lem}

\begin{rmk}
By Markov's inequality and Lemma 1 of \citet{liu2010asymptotics}, one obtains
\begin{equation}
\P(|S_n-S_{n, m}|\geq x) \leq \frac{\|S_n-S_{n,m}\|_q^q}{x^q}
 \le C_q \frac{n^{q/2}m^{-\alpha q}\|e_\cdot\|_{q, \alpha}^q}{x^q}.
\end{equation}
In comparison, the polynomial tail bounds in (\ref{eq:1508220935}) and (\ref{eq:1508220936}) are sharper.
\end{rmk}

\subsection{Inequalities with Finite Exponential Moments}
If $e_i$ satisfies stronger moment condition than the existence of finite $q$-th moment, we can have an exponential inequality. We shall assume $\|e_\cdot\|_{q, \alpha} < \infty$ for all $q >0$ and some $\alpha\geq 0$ and we further assume for some $\nu \geq 0$, the dependence adjusted sub-exponential norm
\begin{equation}
\label{assumption}
\|e_\cdot\|_{\psi_\nu, \alpha}:=\sup\limits_{q \geq 2} q ^{-\nu} \|e_\cdot\|_{q, \alpha} < \infty.
\end{equation}
By this definition, if $e_i$ are i.i.d., $\|e_\cdot\|_{\psi_\nu, \alpha}$ reduces to the sub-Gaussian norm ($\nu=1$) or sub-exponential norm ($\nu=1/2$) of the random variable by the equivalence of $\|e_\cdot\|_{q, \alpha}$ and $\|e_i\|_q$. The parameter $\nu$ measures how fast $\|e_\cdot\|_{q, \alpha}$ increases with $q$.
\begin{lem}
\label{m-approximationtheorem2}
Assume (\ref{assumption}). Let $J_n=(S_n-S_{n, m})/\sqrt{n}$ and $\beta=2/(1+2\nu)$.
Then
\begin{equation*}
h(t):= \sup_{n \in \mathbb{N}} \E[\exp(tJ_n^{\beta})] \leq 1+C_{\beta}(1-t/t_0)^{-1/2} t/t_0
\end{equation*}
holds for $0 \leq t < t_0$ with $t_0=m^{\alpha\beta}/(e\beta\|e_\cdot\|_{\psi_\nu, \alpha}^\beta)$. Consequently, letting $t=t_0/2$, for $x>0$,
\begin{equation}
\label{conclusion1}
\Prob(|J_n| \geq x) \leq \exp(-t x^\beta)h(t) \leq C_\beta\exp\left(-\frac{x^\beta m^{\alpha \beta}}{2e\beta \|e_\cdot\|^\beta_{\psi_\nu, \alpha}}\right).
\end{equation}
\end{lem}

\begin{lem}[cf. Theorem 3 of \citet{wu2014}]
\label{theorem for the sum exponential}
Assume (\ref{assumption}) holds for $\alpha=0$. Let $\beta=2/(1+2\nu)$. Then for $x>0$,
\begin{equation}
\Prob(|S_n/\sqrt{n}| \geq x) \leq C_\beta \exp\left(-\frac{x^\beta}{2e\beta \|e_\cdot\|^\beta_{\psi_\nu, 0}}\right).
\end{equation}
\end{lem}

\begin{proof}[Proof of Lemma \ref{m-approximationtheorem2}]
Let $Q_{n, l}=\sum_{i=1}^n\mathcal{P}_{i-l}X_i$, $l \geq 0$. Then $Q_{n, l}$ is a backward martingale. By Burkholder's inequality, we have
\begin{equation*}
\|Q_{n, l}\|_q^2 \leq (q-1) \sum_{i=1}^n \|\mathcal{P}_{i-l}X_i\|_q^2 =(q-1)n (\theta_{l, q}')^2.
\end{equation*}
By $\theta_{l, q}' \leq \delta_{l, q}$, we have $\|J_n\|_q \leq (q-1)^{1/2}\Delta_{m+1, q}$ in view of $\sqrt{n}J_n= \sum_{l=m+1}^\infty Q_{n, l}$. Write the negative binomial expansion $(1-s)^{-1/2}=1+\sum_{k=1}^\infty a_k s^k$ with $a_k=(2k)!/(2^{2k}(k!)^2)$ for $|s|<1$. By Stirling's formula, we have $a_k \sim (k\pi)^{-1/2}$ as $k \rightarrow \infty$. Hence, there exists absolute constants $C_1, C_2>0$ such that for all $k \geq 1$,
\begin{equation}
C_1 (k/e)^ka_k^{-1} \leq k! \leq C_2 (k/e)^k a_k^{-1}.
\end{equation}
Under condition (\ref{assumption}), if $k\beta \geq 2$, then $\|e_\cdot\|_{\beta k, \alpha} \leq \|e_\cdot\|_{\psi_\nu, \alpha}(\beta k)^\nu$ and hence
\begin{eqnarray*}
\frac{t^k\|J_n^\beta\|^k_k}{k!} \leq \frac{t^k(\beta k-1)^{\beta k/2}\Delta_{m+1,\beta k}^{\beta k}}{C_1(k/e)^k a_k^{-1}}
\leq \frac{a_k t^k (\beta k -1)^{\beta k/2}}{C_1 t_0^k(\beta k)^{\beta k/2}} \leq \frac{a_k t^k}{C_1 \sqrt{e} t_0^k}.
\end{eqnarray*}
If $k\beta <2$, then $\|J_n\|_{\beta k} \leq \|J_n\|_2 \leq 2^\nu m^{-\alpha}\|e_\cdot\|_{\psi_\nu, \alpha}$. In $e^y= \sum_{k=0}^\infty y^k/k!$, let $y = t J_n^\beta$, then
\begin{eqnarray*}
h(t) &\leq & 1+ \sum_{1 \leq k <2/\beta}\frac{t^k(2^\nu m^{-\alpha} \|e_\cdot\|_{\psi_\nu, \alpha})^{\beta k}}{k!}+\sum_{k \geq 2/\beta}\frac{a_k t^k}{C_1 \sqrt{e} t_0^k}\\
&\leq & 1+C_\beta \sum_{k=1}^\infty a_k\frac{t^k}{t_0^k} \leq 1+ C_\beta \frac{t/t_0}{(1-t/t_0)^{1/2}},
\end{eqnarray*}
where $C_\beta>0$ only depends on $\beta$. So (\ref{conclusion1}) follows by Markov's inequality.
\end{proof}

\subsection{Inequalities for High-dimensional Time Series with Finite Polynomial Moments }
In this section we shall derive powerful tail probability inequalities for high-dimensional stationary vectors; cf Lemmas \ref{lemma for m-approximation, polynomial tail} and \ref{lemma for the sum, polynomial tail}. The proofs require Theorem 4.1 of \citet{pinelis1994optimum}, a deep Rosenthal-Burkholder type bound on moments of Banach-spaced martingales. Lemma \ref{lemma for high-dim Bulkholder} follows from Theorem 4.1 of \citet{pinelis1994optimum}. Lemma \ref{lemma for independent high-dim Nagaev} is a Fuk-Magaev type inequality for the sum of independent random vectors. For a $p$-dimensional vector $v = (v_1, \ldots, v_p)$ recall the $s$-length $|v|_s = (\sum_{j=1}^p |v_j|^s)^{1/s}$, $s \ge 1$.

\begin{lem}
\label{lemma for high-dim Bulkholder}
Let $D_i$, $1 \leq i \leq n$, be $p$-dimensional martingale difference vectors with respect to the $\sigma$-field $\mathcal{G}_i$. Let $s>1$ and $q\geq 2$. Then
\begin{equation*}
\||D_1+\ldots+D_n|_{s}\|_{q} \le c \left\{ q \| \sup_{i} |D_i|_s \|_q +\sqrt{q(s-1)}\left\|\left[\sum_{i=1}^n \E (|D_i|_s^2|\mathcal{G}_{i-1})\right]^{1/2}\right\|_q \right\},
\end{equation*}
where $c$ is an absolute constant.
\end{lem}

\begin{lem}
\label{lemma for independent high-dim Nagaev}
Assume $s>1$. Let $X_1, \ldots, X_n$ be $p$-dimensional independent random vectors with mean zero such that for some $q>2$, $\||X_i|_s\|_q < \infty$, $1 \leq i \leq n$. Let $T_n=\sum_{i=1}^n X_i$ and $\sigma_i=(\|X_{i1}\|_2, \ldots, \|X_{ip}\|_2)^\top$. Then for any $y>0$,
\begin{equation}
\P \left(|T_n|_s\geq 2 \E |T_n|_s+y\right) \leq C_q y^{-q}\sum_{i=1}^n \E |X_i|_s^q+\exp\left(-\frac{y^2}{3\sum_{i=1}^n |\sigma_i|_s^2}\right),
\end{equation}
where $C_q$ is a positive constant only depending on $q$.
\end{lem}
\begin{proof}[Proof of Lemma \ref{lemma for independent high-dim Nagaev}]
For $s>1$, we apply Theorem 3.1 of \citet{einmahl2008characterization} with the Banach space $(\mathbb{R}^p, |\cdot|_s)$ and $\eta=\delta=1$. The unit ball of the dual of $(\mathbb{R}^p, |\cdot|_s)$ is the set of linear functions $\{u=(u_1, \ldots, u_p)^\top \mapsto \lambda^\top u: \lambda \in \mathbb{R}^p, |\lambda|_a \leq 1\}$ where $1/a+1/s=1$. By Minkowski's and H\"{o}lder's inequalities, we have
\begin{equation*}
\|\lambda^\top X_i\|_2 \leq \sum_{j=1}^p |\lambda_j| \cdot \|X_{ij}\|_2 \leq |\lambda|_a |\sigma_i|_s.
\end{equation*}
Hence, the $\Lambda_n$ therein is bounded by $\sum_{i=1}^n |\sigma_i|_s^2$.
\end{proof}

Let $X_i$ be a mean zero $p$-dimensional stationary process, and $T_n = \sum_{i=1}^n X_i$, $T_{n,m}=\sum_{i=1}^n X_{i,m}$ where $X_{i,m}= \E(X_i| \eps_{i-m}, \ldots, \eps_{i})$. We are interested in bounding the tail probabilities of $\P (|T_n-T_{n,m}|_\infty \geq x)$ and $\P (|T_n|_\infty \geq x)$ for large $x$. Wrtie $\ell = \ell(p) = 1\vee \log p$.

\begin{lem}
\label{lemma for m-approximation, polynomial tail}
Assume $\| |X_\cdot|_\infty \|_{q, \alpha} < \infty$, where $q>2$ and $\alpha \geq 0$. Also assume $\Psi_{2, \alpha} < \infty$. (i) If $\alpha>1/2-1/q$, then for $x \gtrsim [\sqrt{n \ell}\Psi_{2, \alpha}+n^{1/q} \ell \| |X_{\cdot}|_\infty \|_{q, \alpha}]m^{-\alpha}$,
\begin{equation}
\label{alpha>}
\P(|T_n-T_{n, m}|_\infty \geq x) \leq \frac{C_{q, \alpha} n m^{q/2-1-\alpha q} \ell^{q/2}\||X_\cdot|_\infty\|_{q, \alpha}^q}{x^q}+ C_{q,\alpha} \exp\lp - \frac{C_{q, \alpha} x^2 m^{2 \alpha}}{n \Psi_{2, \alpha}^2}\rp
\end{equation}
holds for all $1 \leq m \leq n$.
(ii) If $0 < \alpha < 1/2-1/q$, the inequality is
\begin{equation}
\label{alpha<}
\P(|T_n-T_{n, m}|_\infty \geq x) \leq \frac{C_{q,\alpha} n^{q/2-\alpha q} \ell^{q/2} \| |X_\cdot|_\infty \|_{q, \alpha}^q}{x^q}+ C_{q, \alpha} \exp\lp - \frac{C_{q, \alpha} x^2 m^{2 \alpha}}{n \Psi_{2, \alpha}^2}\rp.
\end{equation}
\end{lem}

\begin{proof}[Proof of Lemma \ref{lemma for m-approximation, polynomial tail}]
Let $s = \ell = 1\vee \log p$. Then $\P (|T_n-T_{n,m}|_\infty \geq x)$ is equivalent to $\P (|T_n-T_{n,m}|_s \geq x)$, since for any vector $v=(v_1, \ldots, v_p)^\top$, $|v|_\infty \leq |v|_s \leq p^{1/s} |v|_\infty$.
Let $L=\lfloor (\log n - \log m)/ (\log 2)\rfloor$, $\varpi_l=2^l$ if $1 \leq l <L$, $\varpi_L=\lfloor n/m \rfloor$ and $\tau_l=m \cdot \varpi_l$ for $1 \leq l < L$, $\tau_0=m$, $\tau_L=n$. Define $M_{n, l}=T_{n, \tau_l}-T_{n, \tau_{l-1}}$ for $1 \leq l \leq L$ and write
\begin{equation}
\label{sum}
T_n-T_{n, m}= T_n-T_{n, n}+\sum\limits_{l=1}^L M_{n, l}.
\end{equation}
Notice that $T_n-T_{n, n}=\sum\limits_{j=n}^\infty T_{n, j+1}-T_{n, j}$. By Lemma \ref{lemma for high-dim Bulkholder},
\begin{equation*}
 \| |T_n-T_{n, n}|_s \|_q \leq \sum\limits_{j=n}^\infty \| |T_{n, j+1}-T_{n, j}|_s \|_q \leq \sum\limits_{j=n}^\infty C_q (ns)^{1/2}\omega_{j+1, q}=C_q (ns)^{1/2}\Omega_{n+1, q},
\end{equation*}
where $C_q$ is a constant only depending on $q$. By Markov's inequality, we have
\begin{equation}
\label{secondterm}
\Prob(|T_n-T_{n, n}|_s \geq x) \leq \frac{\| |T_n-T_{n, n}|_s \|_q ^q}{x^q} \leq \frac{C_q (ns)^{q/2}\Omega_{n+1, q}^q}{x^q}.
\end{equation}
For each $1 \leq l \leq L$, define
\begin{eqnarray*}
&&Y_{i,l}=\sum_{k=(i-1)\tau_l+1}^{(i\tau_l)\wedge n}\lp X_{k,\tau_l}-X_{k,\tau_{l-1}}\rp, \quad \text{for } 1 \leq i \leq \lfloor n/\tau_l\rfloor;\\
&&R_{n, l}^e= \sum_{i \text{ is even}} Y_{i,l} \text{ and } R_{n, l}^o= \sum_{i \text{ is odd}} Y_{i,l}.
\end{eqnarray*}
Let $c=q/2-1-\alpha q$; let $\lambda_1, \lambda_2, \cdots, \lambda_L$ be a positive sequence such that $\sum_{l=1}^L \lambda_l \leq 1$, specifically, $\lambda_l=l^{-2}/(\pi^2/3)$ if $1 \leq l \leq L/2$ and $\lambda_l=(L+1-l)^{-2}/(\pi^2/3)$ if $L/2 < l \leq L$.
Since $Y_{i,l}$ and $Y_{i',l}$ are independent for $|i-i'|>1$, by Lemma \ref{lemma for independent high-dim Nagaev}, for any $x>0$,
\begin{eqnarray*}
\P( |R_{n,l}^e|_s-2 \E | R_{n,l}^e|_s\geq \lambda_lx) \leq  \frac{C_q \sum\limits_{i \text{ is even}}\E| Y_{i,l}|_s^{q}}{\lp \lambda_l x \rp^{q}}+ \exp\lp -\frac{\lp \lambda_l x \rp^2}{3 \sum\limits_{i \text{ is even}} |\sigma_{Y_i, l}|_s^2} \rp,
\end{eqnarray*}
where $\sigma_{Y_i, l}=(\|Y_{i1, l}\|_2, \ldots, \|Y_{ip,l}\|_2)^\top$.
By Lemma \ref{lemma for high-dim Bulkholder}, $\| |Y_{i, l}|_s \|_{q} \leq C_q (\tau_l s)^{1/2} \tilde{\omega}_{l, q}$ where $\tilde{\omega}_{l, q} =\sum_{k=\tau_{l-1}+1}^{\tau_l} \omega_{k, q} \leq \tau_{l-1}^{-\alpha} \||X_\cdot|_\infty \|_{q, \alpha}$. For $1 \leq j \leq p$, by the Bulkholder inequality, $\|Y_{ij,l}\|_2 \leq \sqrt{\tau_l} \tilde{\delta}_{l,2,j}$ where $\tilde{\delta}_{l,2,j}= \sum_{k=\tau
_{l-1}+1}^{\tau_l} \delta_{k, 2, j} \leq \tau_{l-1}^{-\alpha} \|X_{\cdot j}\|_{2, \alpha}$, which implies $|\sigma_{Y_i, l}|_s \lesssim \tau^{1/2}\tau_{l-1}^{-\alpha} \Psi_{2, \alpha}$. So we obtain
\begin{equation}
\P( | R_{n,l}^e|_s -2\E | R_{n,l}^e|_s \geq \lambda_l x) \leq  \frac{C_1 n s^{q/2}}{x^{q}}\cdot \frac{\tau_{l}^{q/2-1}\tilde{\omega}_{l, q}^{q}}{\lambda_l^{q}}+ \exp\lp -\frac{C_2\lp\lambda_l x \rp^2\tau_{l-1}^{2\alpha}}{n \Psi_{2, \alpha}^2} \rp. \label{R}
\end{equation}
By Lemma 8 of \citet{chernozhukov2014comparison}, for $s = \log p \vee 1$,
\begin{equation}
\E | R_{n,l}^e|_s \lesssim \sqrt{ns} \tau_{l-1}^{-\alpha} \Psi_{2, \alpha} + n^{1/q}s \tilde{\omega}_{l,q} \lesssim  [\sqrt{ns}\Psi_{2, \alpha}+n^{1/q} s \| |X_{\cdot}|_\infty\|_{q, \alpha}]m^{-\alpha} \varpi_l^{-\alpha}.
\end{equation}
Notice that $\min_{l \geq 0} \lambda_l \varpi_l^{\alpha} >0$. Hence, $\E | R_{n,l}^e|_s \lesssim \lambda_l x$ and (\ref{R}) implies
\begin{equation*}
\P( | R_{n,l}^e|_s \geq \lambda_l x) \leq  \frac{C_1 n s^{q/2}}{x^{q}}\cdot \frac{\tau_{l}^{q/2-1}\tilde{\omega}_{l, q}^{q}}{\lambda_l^{q}}+ \exp\lp -\frac{C_2\lp\lambda_l x \rp^2\tau_{l-1}^{2\alpha}}{n \Psi_{2, \alpha}^2} \rp.
\end{equation*}

A similar inequality holds for $R_{n, l}^o$. Therefore,
\begin{eqnarray}
\Prob (|\sum\limits_{l=1}^L M_{n, l}|_s  \geq x ) \nonumber
&\leq& \sum_{l=1}^L \P\lp | M_{n,l}|_s \geq \lambda_l x \rp\\ \nonumber
&\leq& \sum_{l=1}^L \P\lp \lal  R_{n,l}^e \ral_s \geq \lambda_l x/2 \rp+\sum_{l=1}^L \P \lp \lal R_{n,l}^o \ral_s \geq \lambda_l x/2 \rp\\ \nonumber
&\leq& \sum_{l=1}^L \frac{C_1 n s^{q/2}}{x^{q}}\cdot \frac{\tau_{l}^{q/2-1}\tilde{\omega}_{l, q}^{q}}{\lambda_l^{q}}+ 2\sum_{l=1}^L\exp\lp -\frac{C_2\lp\lambda_l x \rp^2\tau_{l-1}^{2\alpha}}{n \Psi_{2, \alpha}^2} \rp \\ \label{thirdterm}
&\leq & \frac{C_3 n m^{c}s^{q/2}\| |X_\cdot|_\infty \|_{q, \alpha}^q}{x^q} \sum_{l=1}^L \frac{\varpi_l^c}{\lambda_l^q}+C_4 \sum_{l=1}^L \exp\left(-\frac{C_5 x^2 m^{2 \alpha} \lambda_l^2 \varpi_l^{2\alpha}}{n \Psi_{2, \alpha}^2}\right).
\end{eqnarray}
By the definition of $\varpi_l$ and $\lambda_l$ and by some elementary calculation, there exists some constant $C_6>1$ such that for all $t \geq 1$,
\begin{equation}
\label{exponentialterm}
\sum_{l=1}^L \exp(-C_5 t \lambda_l^2 \varpi_l^{2\alpha}) \leq C_6 \exp(-C_5 t \mu),
\end{equation}
where $\mu=\min_{l \geq 1} \lambda_l^2 \varpi_l^{2 \alpha} >0$. If $c>0$, it can be obtained that $\sum_{l=1}^L \varpi_l^c/ \lambda_l^q \leq C_7 \varpi_L^c \leq C_7 n^c/m^c$. If $c <0$, then $\sum_{l=1}^L \varpi_l^c/ \lambda_l^q \leq C_8$. Hence, combining (\ref{sum}), (\ref{secondterm}), (\ref{thirdterm}), (\ref{exponentialterm}), Lemma   \ref{lemma for m-approximation, polynomial tail} follows.
\end{proof}

\begin{lem}
\label{lemma for the sum, polynomial tail}
Assume $\| |X_\cdot|_\infty \|_{q, \alpha} < \infty$, where $q>2$ and $\alpha \geq 0$. Also assume $\Psi_{2, \alpha} < \infty$. (i) If $\alpha>1/2-1/q$, then for $x \gtrsim \sqrt{n \ell}\Psi_{2, \alpha}+n^{1/q} \ell \| |X_{\cdot}|_\infty \|_{q, \alpha}$,
\begin{equation}
\label{alpha> for the sum}
\P(|T_n|_\infty \geq x) \leq \frac{C_{q, \alpha} n \ell^{q/2} \| |X_\cdot|_\infty \|_{q, \alpha}^q}{x^q}+ C_{q,\alpha} \exp\lp - \frac{C_{q, \alpha} x^2}{n \Psi_{2, \alpha}^2}\rp.
\end{equation}
(ii) If $0 < \alpha < 1/2-1/q$, we have the following inequality,
\begin{equation}
\label{alpha< for the sum}
\P(|T_n|_\infty \geq x) \leq \frac{C_{q,\alpha} n^{q/2-\alpha q} \ell^{q/2} \| |X_\cdot|_\infty \|_{q, \alpha}^q}{x^q}+ C_{q, \alpha} \exp\lp - \frac{C_{q, \alpha} x^2}{n \Psi_{2, \alpha}^2}\rp.
\end{equation}
\end{lem}

\begin{proof}[Proof of Lemma \ref{lemma for the sum, polynomial tail}]
The proof is similar to that of Lemma \ref{lemma for m-approximation, polynomial tail}, and thus is omitted.
\end{proof}

\section{Proofs}
\label{sec:proof}

\subsection{Proof of Theorems \ref{th:1507140259} and \ref{th:1507140325}}
\label{sec:1507151006}

We shall apply the $m$-dependence approximation approach. For $m \ge 0$, define
\begin{equation}
X_{i, m}=(X_{i1,m}, \ldots, X_{ip, m})^\top=\E(X_i|\eps_{i-m},\eps_{i-m+1}, \ldots, \eps_i).
\end{equation}
Write $T_{X}=\sum_{i=1}^n X_i$ and $T_{X, m}=\sum_{i=1}^n X_{i, m}$.
For simplicity, suppose $n=(M+m)w$, where $M \gg m$ and $M, m, w \rightarrow \infty$ (to be determined) as $n \to \infty$. We apply the block technique and split the interval $[1, n]$ into alternating large blocks $L_b= [ (b-1)(M+m)+1,  bM+(b-1)m ]$ and small blocks $S_b =[ bM+(b-1)m +1, b(M+m) ]$, $1 \leq b \leq w$. Let
\begin{eqnarray*}
Y_b = \sum_{i \in L_b}X_i, \,\, Y_{b, m}=\sum_{i \in L_b} X_{i, m}, \,\,
 T_{Y}=\sum_{b=1}^w Y_{b}, \,\, T_{Y, m}=\sum_{b=1}^w Y_{b, m}.
\end{eqnarray*}
Let $Z_{b}$, $1 \leq b \leq w$, be i.i.d. $N(0, M B)$ and $Z_{b, m}$ be i.i.d. $N(0, M \tilde B)$, where the covariance matrices $B$ and $\tilde B$ are respectively given by
\begin{eqnarray}
\label{eq:1507172300}
B = (b_{ij})_{i,j=1}^p = \mbox{Cov}(Y_b/\sqrt{M}) \mbox{ and }
\tilde{B}=(\tilde{b}_{ij})_{i,j=1}^p = \mbox{Cov}(Y_{b, m}/\sqrt{M}).
\end{eqnarray}
Write $T_{Z, m}=\sum_{b=1}^w Z_{b, m}$ and let $Z \sim N(0,\Sigma)$.

\begin{lem}
\label{lemmafordifference}
(i) Assume $\Theta_{q, \alpha}< \infty$ for some $q>2$ and $\alpha>0$. Then there exists some constant $C_{q, \alpha}$ such that for $y > 0$
\begin{equation}\label{eq:1508122226}
\P(|T_X-T_{Y, m}|_\infty \geq y) \lesssim f^*_1(y)+f^*_2(y)=: f^*(y)
\end{equation}
where the constant in $ \lesssim$ only depends on $q$ and $\alpha$,
\begin{equation}
f^*_1(y)=\left\{
\begin{array}{ll}
  y^{-q}n m^{q/2-1-\alpha q}\Theta_{q, \alpha}^q+ p\exp\lp - \frac{C_{q, \alpha}y^2 m^{2 \alpha}}{n\Psi_{2, \alpha}^2}\rp, & \alpha>1/2-1/q \\
  y^{-q}n^{q/2-\alpha q} \Theta_{q, \alpha}^q+  p \exp\lp - \frac{C_{q, \alpha}y^2 m^{2 \alpha}}{n\Psi_{2, \alpha}^2}\rp, & \alpha<1/2-1/q
\end{array}
\right.
\end{equation}
and
\begin{equation}
f^*_2(y)=\left\{
\begin{array}{ll}
  y^{-q}wm \Theta_{q, \alpha}^q+  p\exp\lp - \frac{C_{q, \alpha}y^2 }{mw\Psi_{2, \alpha}^2}\rp, & \alpha>1/2-1/q \\
  y^{-q}(wm)^{q/2-\alpha q} \Theta_{q, \alpha}^q+ p\exp\lp - \frac{C_{q, \alpha} y^2 }{wm\Psi_{2, \alpha}^2}\rp, & \alpha<1/2-1/q
\end{array}
\right. .
\end{equation}
(ii) Assume $\Phi_{\psi_\nu, \alpha}<\infty$ for some $\nu\geq 0$ and $\alpha>0$. Let $\beta=2/(1+2\nu)$. Then there exists a constant $C_\beta >0$ such that for $y>0$,
\begin{equation}\label{eq:1508122228}
\P(|T_X-T_{Y, m}|_\infty \geq y) \lesssim f^\diamond_1(y)+ f^\diamond_2(y)=: f^\diamond(y),
\end{equation}
where the constant in $ \lesssim$ only depends on $\beta$ and $\alpha$,
\begin{equation*}
f^\diamond_1(y)=   p\exp\left\{-C_\beta\left(\frac{ y m^{\alpha}}{\sqrt{n}\Phi_{\psi_\nu, \alpha}}\right)^\beta\right\} \,\,
 \text{and} \,\, f^\diamond_2(y)= p\exp\left\{-C_\beta\left(\frac{ y }{\sqrt{mw} \Phi_{\psi_\nu, 0}}\right)^\beta\right\}.
\end{equation*}
\end{lem}
\begin{proof}
Let $P_1 = \P(|T_X-T_{X, m}|_\infty \geq y/2)$ and $P_2 = \P(|T_{X, m}-T_{Y, m}|_\infty \geq y/2)$. Lemmas \ref{m-approximationtheorem1} and \ref{lemma for m-approximation, polynomial tail} imply that $P_1 \leq f^*_1(y)$. Write $T_{X, m}-T_{Y, m} = \sum_{b=1}^w \sum_{i \in S_b} X_{i, m}$. By Lemmas \ref{theorem for the sum polynomial} and \ref{lemma for the sum, polynomial tail}, we also have $P_2 \leq f^*_2(y)$. Hence both cases with $\alpha>1/2-1/q$ and $\alpha<1/2-1/q$ of Lemma \ref{lemmafordifference}(i) follow in view of $\P(|T_X-T_{Y, m}|_\infty \geq y) \leq P_1+P_2$.

The exponential moment case (ii) similarly follows from $P_1 \leq f^\diamond_1(y)$ and $P_2 \leq f^\diamond_2(y)$.
\end{proof}

\begin{lem}
\label{mainlemma}
Let $D=(d_{i j})_{i, j=1}^p$ be a diagonal matrix. Assume that there exist constants $c>0, c_2>c_1>0$ such that $c<\min_{1 \leq j \leq p} d_{jj}$ and $c_1 \leq \tilde{b}_{jj}/d_{jj} \leq c_2$ for all $1 \leq j \leq p$. Assume $\Psi_{q, 0}< \infty$ for some $q \geq 4$. Then for all $\lambda \in (0,1)$,
\begin{eqnarray*}
&&\sup\limits_{t \in \mathbb{R}} \left| \P(|D^{-1/2}T_{Y, m}/\sqrt{n}|_\infty \leq t)-\P(|D^{-1/2}T_{Z, m}/\sqrt{n}|_\infty \leq t)  \right| \\
&\lesssim &  w^{-1/8}(\Psi_{3, 0}^{3/4}\vee\Psi_{4, 0}^{1/2})(\log(pw/\lambda))^{7/8} +w^{-1/2}(\log(pw/\lambda))^{3/2} u_m(\lambda)+\lambda\\
&=:& h(\lambda, u_m(\lambda)),
\end{eqnarray*}
where the constant in $\lesssim$ depends on $c, c_1, c_2$, and $q$ and $\alpha$ for (i), and $\beta$ for (ii) below, and $u_m(\lambda) \le u^*_m(\lambda)$ in (i), and $u_m(\lambda) \le u^\diamond_m(\lambda)$ in (ii).

(i) Assume $\Theta_{q, \alpha}< \infty$ for some $q \geq 4$ and $\alpha>0$, then
\begin{equation}
u^*_m(\lambda) = \left\{
\begin{array}{ll}
 \max\{\Theta_{q, \alpha}(\lambda^{-1}w)^{1/q}M^{1/q-1/2}, \Psi_{2, \alpha}\sqrt{\log(pw/\lambda)}\}, & \alpha>1/2-1/q \\
 \max\{\Theta_{q, \alpha}(\lambda^{-1}w)^{1/q}M^{-\alpha}, \Psi_{2, \alpha}\sqrt{\log(pw/\lambda)}\}, & \alpha<1/2-1/q.
\end{array} \right.
\end{equation}

(ii) Assume $\Phi_{\psi_\nu, 0}<\infty$ for some $\nu\geq 0$. Then
\begin{equation}
u^\diamond_m(\lambda) = \max\{\Phi_{\psi_\nu, 0}(\log (pw/\lambda))^{1/\beta}, \sqrt{\log (pw/\lambda)}\}.
\end{equation}
\end{lem}

\begin{proof}
For $1<l\leq q$, define $R_l=\max_{1 \leq j \leq p} \| M^{-1/2} Y_{bj, m}\|_l$.
Since $X_{ij, m}=\sum_{k=0}^m \mathcal{P}_{i-k}X_{ij}$, by Burkholder's inequality (\citet{burkholder1973}),
\begin{equation*}
\|\sum_{i=1}^M \mathcal{P}_{i-k}X_{ij}\|_l^2 \leq C_l \sum_{i=1}^M \|\mathcal{P}_{i-k}X_{ij}\|_l^2 \leq C_l M (\theta'_{k, l, j})^2,
\end{equation*}
then we have
\begin{equation}
\label{moment for the sum}
\|\sum_{i=1}^M X_{ij, m}\|_l \leq C_l \sum_{k=0}^m \|\sum_{i=1}^M \mathcal{P}_{i-k}X_{ij}\|_l \leq C_l M^{1/2} \Delta_{0, l, j},
\end{equation}
which implies $R_l \leq C_l \Psi_{l, 0}$.
For $0<\lambda<1$ and the diagonal matrix $D=(d_{ij})_{i,j=1}^p$, define $u_{Y, m}(\lambda)$ as the infimum over all numbers $u>0$ such that
\begin{equation*}
\P(|M^{-1/2} d_{jj}^{-1/2}Y_{bj, m}| \leq u, 1 \leq b \leq w, 1\leq j \leq p) \geq 1-\lambda.
\end{equation*}
Also define $u_{Z, m}(\lambda)$ by the corresponding quantity for the analogue Gaussian case, namely with $Y_{b, m}$ replaced by $Z_{b, m}$ in the above definition. Let $u_{m}(\lambda):= u_{Y, m}(\lambda) \vee u_{Z, m}(\lambda)$.
By Theorem 2.2 of \citet{chernozhukov2013}, for all $\lambda \in (0,1)$,
\begin{eqnarray*}
&&\sup\limits_{t \in \mathbb{R}} \left| \P(|D^{-1/2}T_{Y, m}/\sqrt{n}|_\infty \leq t)-\P(|D^{-1/2}T_{Z, m}/\sqrt{n}|_\infty \leq t)  \right| \\
&\lesssim & w^{-1/8}(R_3^{3/4}\vee R_4^{1/2})(\log(pw/\lambda))^{7/8}+w^{-1/2}(\log(pw/\lambda))^{3/2} u_m(\lambda)+\lambda,
\end{eqnarray*}
Now we shall find a bound on the function $u_{m}(\lambda)$.
(i) By Lemmas \ref{theorem for the sum polynomial} and \ref{lemma for the sum, polynomial tail}, we have
\begin{eqnarray*}
&&\P(|M^{-1/2} d_{jj}^{-1/2} Y_{bj, m}| > u \text{ for some } b, j) \leq \P(|M^{-1/2}Y_{b, m}|_\infty > c^{1/2}u) \\
&&\leq \left\{
\begin{array}{ll}
C_{q, \alpha} u^{-q}wM^{1-q/2} \Theta_{q, \alpha}^q+ C_{q,\alpha} pw \exp\lp -\frac{C_{q, \alpha}u^2}{\Psi_{2,\alpha}^2}\rp, & \alpha>1/2-1/q \\
C_{q,\alpha} u^{-q}w M^{-\alpha q}\Theta_{q, \alpha}^q + C_{q, \alpha} pw \exp\lp - \frac{C_{q, \alpha} u^2 }{\Psi_{2,\alpha}^2}\rp, & \alpha<1/2-1/q
\end{array}
\right. .
\end{eqnarray*}
This implies $u_{Y, m}(\lambda) \leq C_{q,\alpha} \max\{\Theta_{q, \alpha}(\lambda^{-1}w)^{1/q}M^{1/q-1/2}, \Psi_{2, \alpha}\sqrt{\log(pw/\lambda)}\}$ if $\alpha>1/2-1/q$ and $u_{Y, m}(\lambda) \leq C_{q, \alpha} \max\{\Theta_{q, \alpha}(\lambda^{-1}w)^{1/q}M^{-\alpha}, \Psi_{2, \alpha}\sqrt{\log(pw/\lambda)}\}$ if $\alpha<1/2-1/q$. For $u_{Z, m}(\lambda)$, since  $M^{-1/2} Z_{bj, m} \sim N(0, \tilde{b}_{jj})$, we have $\E(\exp\{M^{-1}Z^2_{bj, m}/(4\tilde{b}_{jj})\})\leq C$. Hence
\begin{eqnarray}
\P(|M^{-1/2}d_{jj}^{-1/2} Z_{bj, m}| > u \text{ for some } b, j) &\leq& \sum_{b=1}^w \sum_{j=1}^p\P(|M^{-1/2}Z_{bj, m}| > d_{jj}^{1/2}u) \nonumber \\ \label{gaussianboundforu}
&\leq & Cpw \exp(-d_{jj}u^2/(4\tilde{b}_{jj})).
\end{eqnarray}
With the assumption $c_1 \leq \tilde{b}_{jj}/d_{jj} \leq c_2$,  $u_{Z, m}(\lambda) \leq C \sqrt{\log(pw/\lambda)}$.\\
(ii)  By Bonferroni inequality and Lemma \ref{theorem for the sum exponential},
\begin{equation}
\label{exponentialboundforu}
\P(|M^{-1/2}d_{jj}^{-1/2} Y_{bj, m}| > u \text{ for some } b, j)
\leq C_\beta pw \exp\left\{-C_\beta\frac{u^\beta}{\Phi^\beta_{\psi_\nu, 0}}\right\},
\end{equation}
where $\beta=2/(1+2\nu)$ and $C_\beta$ is a constant that depends on $\beta$ only. Combining (\ref{gaussianboundforu}) and (\ref{exponentialboundforu}), it follows that $u_{m}(\lambda)\leq C_{\beta}\max\{\Phi_{\psi_\nu, 0}(\log (pw/\lambda))^{1/\beta}, \sqrt{\log (pw/\lambda)}\}$.
\end{proof}

Now we consider the comparison between $Z$ and $T_{Z, m}$. Let $\pi(x)= x^{1/3}(1 \vee \log(p/x))^{2/3}$ for $x>0$.
\begin{lem}
\label{lemmafortwogaussian}
Assume $\Psi_{2, \alpha}<\infty$ for some $\alpha>0$. Let $D=(d_{ij})_{i, j=1}^p$ be a diagonal matrix such that there exist some constants $0< C_1< C_2$ such that $C_1 \leq \sigma_{jj}/d_{jj} \leq C_2$ for all $1 \leq j \leq p$. Then we have
\begin{eqnarray*}
\label{Tz,m-Tz}
&&\sup\limits_{t \in \mathbb{R}} \left| \P(|D^{-1/2}T_{Z, m}/\sqrt{n}|_\infty \leq t)-\P(|D^{-1/2}Z|_\infty \leq t)  \right| \\
&\lesssim & \pi(\max_{1\leq j \leq p}d^{-1}_{jj}\Psi_{2, \alpha} \Psi_{2, 0} (m^{-\alpha}+v(M))+wm/n),
\end{eqnarray*}
where $v(M)$ is the same as defined in Corollary \ref{corollary for Sigma}.
\end{lem}
\begin{proof}
By the definition of $T_{Z, m}$ and $Z$ and (\ref{eq:1507172300}),
\begin{eqnarray*}
&&\Sigma^{Z, m}:=\Cov(D^{-1/2}T_{Z, m}/\sqrt{n})= \frac{Mw}{n}D^{-1/2}\tilde{B}D^{-1/2}, \\ &&\Sigma^{Z}:=\Cov(D^{-1/2}Z)= D^{-1/2}\Sigma D^{-1/2}.
\end{eqnarray*}
Let $S_{M j} = \sum_{i=1}^M X_{ij}$ and $S_{M j, m} = \sum_{i=1}^M X_{ij, m}$. By the moment inequality in \citet{Wu2005}, $ \| S_{M j} \|_2 \leq M^{1/2} \Delta_{0, 2, j}$, $\| S_{M j, m}\|_2 \leq M^{1/2} \Delta_{0, 2, j}$ and  $\|S_{M j} - S_{M j, m}  \|_2 \leq M^{1/2} \Delta_{m+1, 2, j}$. Note that $b_{jk}=M^{-1}\E(S_{M j} S_{M k})$ and $\tilde{b}_{jk}=M^{-1}\E(S_{M j, m}  S_{M k, m})$. Then
\begin{eqnarray*}
|b_{jk}-\tilde{b}_{jk}|&=& \frac{1}{M} |\E (S_{M j} S_{M k}  - S_{M j, m}  S_{M k, m})| \\
&\leq & \frac{1}{M}\left(\|S_{M j}  \|_2\cdot\|S_{M k}  - S_{M k, m}\|_2+\|S_{M k, m} \|_2\cdot\| S_{M j} - S_{M j, m}  \|_2\right) \cr
&\leq & 2\Psi_{2, \alpha}\Psi_{2, 0} m^{-\alpha}.
\end{eqnarray*}
Recall that $\sigma_{jk}=\sum_{l=-\infty}^\infty \gamma_{jk}(l)$ and
\begin{equation*}
b_{jk}=M^{-1}\E(S_{Mj} S_{Mk})=M^{-1}\sum_{l=-M}^M (M-|l|)\gamma_{jk}(l).
\end{equation*}
It follows that
\begin{equation*}
\sigma_{jk}-b_{jk}=\sum_{|l|> M} \gamma_{jk}(l) +M^{-1}\sum_{l=-M}^M |l| \gamma_{jk}(l).
\end{equation*}
By $X_{ij}=\sum_{h=0}^\infty \mathcal{P}^{i-h}X_{ij}$, we have
\begin{equation*}
|\gamma_{jk}(l)|=|\sum_{h=0}^\infty \E[(\mathcal{P}^{-h}X_{0j})(\mathcal{P}^{-h}X_{lk})]| \leq \sum_{h=0}^\infty |\E[(\mathcal{P}^{-h}X_{0j})(\mathcal{P}^{-h}X_{lk})]| \leq \sum_{h=0}^\infty \delta_{h, 2, j}\delta_{h+l, 2, k}.
\end{equation*}
Hence, it can be obtained that
\begin{equation*}
\left|\sum_{|l|>M} \gamma_{jk}(l)\right| \leq 2 \sum_{l=M+1}^\infty |\gamma_{jk}(l)| \leq 2\sum_{l=M+1}^\infty \sum_{h=0}^\infty \delta_{h, 2, j}\delta_{h+l, 2, k} \leq 2 \Delta_{0, 2, j}\Delta_{M+1, 2, k},
\end{equation*}
and
\begin{equation*}
\left|\frac{1}{M}\sum_{l=-M}^M |l| \gamma_{jk}(l)\right| \leq \frac{2}{M} \sum_{l=1}^M \sum_{\iota=k}^M \sum_{h=0}^\infty \delta_{h, 2, j}\delta_{h+\iota, 2, k} \leq \frac{2}{M} \Delta_{0, 2, j} \sum_{l=1}^M \Delta_{l, 2, k}.
\end{equation*}
Since $\Delta_{0, 2, j} \leq \Psi_{2,0}$ and $\Delta_{m, 2, j} \leq  \Psi_{2, \alpha}m^{-\alpha}$, $\max_{1 \leq j,k \leq p}|b_{jk}-\sigma_{jk}|\leq \Psi_{2,\alpha}\Psi_{2,0}v(M)$. Hence,
\begin{eqnarray*}
|\Sigma^{Z, m}-\Sigma^{Z}|_\infty &\leq & \max_{1\leq j \leq p}d^{-1}_{jj}(|\tilde{B}-B|_\infty+ |B-\Sigma|_{\infty})+(1-Mw/n)|D^{-1/2}\Sigma D^{-1/2}|_\infty\\
&\leq &\max_{1\leq j \leq p}d^{-1}_{jj}\Psi_{2, \alpha} \Psi_{2, 0} (m^{-\alpha}+v(M))+ C_2 wm/n.
\end{eqnarray*}
By Theorem 2 of \citet{chernozhukov2014comparison}, the result follows.
\end{proof}

\begin{thm}
\label{maintheorem}
 Let $\Sigma_0$ be the diagonal matrix of the long run covariance matrix $\Sigma$ and $D_0=\Sigma_0^{1/2}$. Let Assumption \ref{assumption:1507172202} be satisfied. (i) Assume that $\Theta_{q, \alpha} < \infty$ holds with some $q \geq 4$ and $\alpha>0$. Then for every $\lambda \in (0, 1)$ and $\eta>0$,
\begin{eqnarray}
\rho_n&:=&\sup\limits_{t \in \mathbb{R}} \left| \P(|D_0^{-1}T_{X}/\sqrt{n}|_\infty \leq t)-\P(|D_0^{-1}Z|_\infty \leq t)  \right| \nonumber \\ \label{maintheorem1}
&\lesssim &  f^*(\sqrt{n}\eta)+ \eta \sqrt{\log p}+ h(\lambda, u^*_m(\lambda))+\pi(\Psi_{2, \alpha}\Psi_{2, 0} (m^{-\alpha}+v(M))+wm/n).
\end{eqnarray}
(ii) Assume $\Phi_{\psi_\nu, \alpha} < \infty$ for some $\nu \geq 0$ and $\alpha>0$. Then for every $\lambda \in (0, 1)$ and $\eta>0$,
\begin{eqnarray}
\rho_n&:=&\sup\limits_{t \in \mathbb{R}} \left| \P(|D_0^{-1}T_{X}/\sqrt{n}|_\infty \leq t)-\P(|D_0^{-1}Z|_\infty \leq t)  \right| \nonumber \\ \label{maintheorem2}
&\lesssim & f^\diamond(\sqrt{n}\eta)+ \eta \sqrt{\log p}+ h(\lambda, u^\diamond_m(\lambda))+\pi(\Psi_{2, \alpha}\Psi_{2, 0} (m^{-\alpha}+v(M))+wm/n).
\end{eqnarray}
\end{thm}
\begin{proof}
(i)
By Lemma \ref{mainlemma} (i) and Lemma \ref{lemmafortwogaussian}, we have for every $\lambda \in (0, 1)$,
\begin{eqnarray}
&&\sup\limits_{t \in \mathbb{R}} \left| \P(|D_0^{-1}T_{Y, m}/\sqrt{n}|_\infty \leq t)-\P(|D_0^{-1}Z|_\infty \leq t)  \right| \nonumber  \\ &\lesssim & h(\lambda, u^*_m(\lambda))+\pi(\Psi_{2, \alpha}\Psi_{2, 0} (m^{-\alpha}+v(M))+wm/n). \label{sup_part1}
\end{eqnarray}
Observe that each component of the Gaussian vector $D_0^{-1}Z$ has variance 1. By Theorem 3 of \citet{chernozhukov2014comparison}, for every $\eta>0$,
\begin{equation}
\label{anticoncentrationforSZ}
\sup\limits_{t \in \mathbb{R}} \P(\left||D_0^{-1}Z|_\infty-t \right|\leq \eta) \lesssim \eta \sqrt{\log p}.
\end{equation}
By the triangle inequality, for every $\eta>0$, we have
\begin{eqnarray*}
&&\sup\limits_{t \in \mathbb{R}} \left| \P(|D_0^{-1}T_{X}/\sqrt{n}|_\infty > t)-\P(|D_0^{-1}T_{Y, m}/\sqrt{n}|_\infty > t)  \right|\\
&\leq& \P(|D_0^{-1}(T_{X}-T_{Y, m})/\sqrt{n}|_{\infty}>\eta)+\sup\limits_{t \in \mathbb{R}} \P(\left||D_0^{-1}T_{Y, m}/\sqrt{n}|_\infty-t \right|\leq \eta),
\end{eqnarray*}
which implies Theorem \ref{maintheorem} (i) in view of Lemma \ref{lemmafordifference} (i), (\ref{sup_part1}) and (\ref{anticoncentrationforSZ}).

(ii) Inequality (\ref{maintheorem2}) can be obtained by replacing $f^*$ and $u^*_m$ with $f^\diamond$ and $u^\diamond_m$ in the above proof.
\end{proof}

\subsection{Proof of Theorem \ref{th:1507140259}}

\begin{proof}
Recall (\ref{eq:1508122226}) for $f^*(\cdot)$. By Theorem \ref{maintheorem}, for $\alpha > 1/2-1/q$, to have (\ref{eq:J021209}), we need
\begin{equation}
\label{b}  \pi(\Psi_{2, \alpha}\Psi_{2, 0} (m^{-\alpha}+v(M))+wm/n) \rightarrow 0
\end{equation}
and for some $\eta>0$ and $\lambda \in (0,1)$,
\begin{eqnarray}
\label{c} && f^*(\sqrt{n}\eta)+\eta \sqrt{\log p}\rightarrow 0,\\
\label{d} && h(\lambda, u^*_m(\lambda))\rightarrow 0.
\end{eqnarray}
Firstly, (\ref{b}) requires $m \gg L_2$, $wm \ll n(\log p)^{-2}$, $w \ll n(\log p)^{-2} (\Psi_{2, \alpha}\Psi_{2,0})^{-1}$ if $\alpha>1$ and $w \ll n/L_2$ if $0<\alpha <1$. Moreover, (\ref{c}) requires $m \gg \max(L_1, (\Psi_{2, \alpha} \log p)^{1/\alpha})$ and $w m \ll \min (N_1, N_2)$. And (\ref{d}) needs (\ref{eq:1507142040}) and $w \gg \max(W_1, W_2)$. We also need $M\asymp n/w \gg m$. Notice that $(\Psi_{2, \alpha} \log p)^{1/\alpha} \lesssim L_2$, $N_2 \lesssim n(\log p)^{-2}$ and $N_2 \leq n(\log p)^{-2}(\Psi_{2, \alpha}\Psi_{2,0})^{-1}$. If
\begin{equation}
\label{e}
\max(L_1, L_2) \max(W_1, W_2) = o(1) \min(n, N_1, N_2),
\end{equation}
then we can always choose $m$ and $w$ such that (\ref{eq:J021209}) holds. Observe that $N_2 \lesssim n$, then (\ref{e}) is reduced to (\ref{eq:1507142041}).

For $0< \alpha < 1/2-1/q$, the function $f^*$ in (\ref{c}) is replaced by $f^\diamond$ (cf. (\ref{eq:1508122228})), which implies $\Theta_{q, \alpha} (\log p)^{1/2} = o(n^{\alpha})$, $m \gg (\Psi_{2, \alpha} \log p)^{1/\alpha} $ and $wm \ll \min(N_2, N_3)$. And $u^*_m$ in (\ref{d}) is replaced by $u_m^\diamond$, implying $w \gg \max(W_1, W_2, W_3)$. By the similar argument, if (\ref{eq:1508010523}) is further assumed, then (\ref{eq:J021209}) also holds for the case $0<\alpha < 1/2-1/q$.
\end{proof}

\begin{rmk}
In the proof of Theorem \ref{th:1507140259}, we exclude the case $\alpha=1$ when $\alpha>1/2-1/q$. If $\alpha=1$, we need to impose the additional assumption
\begin{equation}
\label{add}
\max(W_1, W_2)=o(n/(L_2\log n))
\end{equation}
to ensure (\ref{b}). The above condition is very mild since (\ref{eq:1507142041}) implies $\max(W_1, W_2) = o(n/L_2)$. If $\log n \lesssim (\log p)^2 \Psi_{2, \alpha}^2$, which trivially holds in the high-dimensional case $p \asymp n^\kappa$ with some $\kappa > 0$, we have $N_2= O(n/\log n)$ and hence (\ref{eq:1507142041}) implies (\ref{add}). Similarly, it is further assumed $\max(W_1, W_4)=o(n/(L_2\log n))$ in Theorem \ref{th:1507140325} if $\alpha=1$.
\end{rmk}

\bibliographystyle{plainnat}
\bibliography{Reference}

\begin{thebibliography}{38}
\providecommand{\natexlab}[1]{#1}
\providecommand{\url}[1]{\texttt{#1}}
\expandafter\ifx\csname urlstyle\endcsname\relax
  \providecommand{\doi}[1]{doi: #1}\else
  \providecommand{\doi}{doi: \begingroup \urlstyle{rm}\Url}\fi

\bibitem[Alexopoulos and Goldsman(2004)]{Alexopoulos2004}
Christos Alexopoulos and David Goldsman.
\newblock To batch or not to batch?
\newblock \emph{ACM Trans. Model. Comput. Simul.}, 14\penalty0 (1):\penalty0
  76--114, 2004.

\bibitem[Anderson(1971)]{anderson1971}
T.W. Anderson.
\newblock \emph{The Statistical Analysis of Time Series}.
\newblock Wiley, 1971.

\bibitem[Bradley(2007)]{bradley2007introduction}
R.C. Bradley.
\newblock \emph{Introduction to Strong Mixing Conditions}.
\newblock Kendrick Press, 2007.

\bibitem[Brockwell and Davis(1991)]{brockwell1991}
P.J. Brockwell and R.A. Davis.
\newblock \emph{Time Series: Theory and Methods}.
\newblock Springer, 1991.

\bibitem[B{\"u}hlmann(2002)]{buhlmann2002}
Peter B{\"u}hlmann.
\newblock Bootstraps for time series.
\newblock \emph{Statistical Science}, 17\penalty0 (1):\penalty0 52--72, 05
  2002.

\bibitem[Burkholder(1973)]{burkholder1973}
D.~L. Burkholder.
\newblock Distribution function inequalities for martingales.
\newblock 1\penalty0 (1):\penalty0 19--42, 02 1973.

\bibitem[Chen et~al.(2015)Chen, Shao, and Wu]{csw2015}
Xiaohong Chen, Qi-Man Shao, and Wei~Biao Wu.
\newblock Self-normalized cramer type moderate deviations under dependence.
\newblock \emph{Manuscript}, 2015.

\bibitem[Chen et~al.(2013)Chen, Xu, and Wu]{chen2013}
Xiaohui Chen, Mengyu Xu, and Wei~Biao Wu.
\newblock Covariance and precision matrix estimation for high-dimensional time
  series.
\newblock \emph{The Annals of Statistics}, 41\penalty0 (6):\penalty0
  2994--3021, 12 2013.

\bibitem[Chernozhukov et~al.(2013{\natexlab{a}})Chernozhukov, Chetverikov, and
  Kato]{chernozhukov2013}
Victor Chernozhukov, Denis Chetverikov, and Kengo Kato.
\newblock Gaussian approximations and multiplier bootstrap for maxima of sums
  of high-dimensional random vectors.
\newblock \emph{The Annals of Statistics}, 41\penalty0 (6):\penalty0
  2786--2819, 12 2013{\natexlab{a}}.

\bibitem[Chernozhukov et~al.(2013{\natexlab{b}})Chernozhukov, Chetverikov, and
  Kato]{chernozhukov2013testing}
Victor Chernozhukov, Denis Chetverikov, and Kengo Kato.
\newblock Testing many moment inequalities.
\newblock \emph{arXiv preprint arXiv:1312.7614}, 2013{\natexlab{b}}.

\bibitem[Chernozhukov et~al.(2014)Chernozhukov, Chetverikov, and
  Kato]{chernozhukov2014comparison}
Victor Chernozhukov, Denis Chetverikov, and Kengo Kato.
\newblock Comparison and anti-concentration bounds for maxima of gaussian
  random vectors.
\newblock \emph{Probability Theory and Related Fields}, 162\penalty0
  (1-2):\penalty0 47--70, 2014.

\bibitem[Dedecker et~al.(2007)Dedecker, Doukhan, Lang, Rafael, Louhichi, and
  Prieur]{dedecker2007weak}
J{\'e}r{\^o}me Dedecker, Paul Doukhan, Gabriel Lang, Le{\'o}n R~Jos{\'e}
  Rafael, Sana Louhichi, and Cl{\'e}mentine Prieur.
\newblock \emph{Weak Dependence: With Examples and Applications}.
\newblock Springer, 2007.

\bibitem[Einmahl and Li(2008)]{einmahl2008characterization}
Uwe Einmahl and Deli Li.
\newblock Characterization of lil behavior in banach space.
\newblock \emph{Transactions of the American Mathematical Society},
  360\penalty0 (12):\penalty0 6677--6693, 2008.

\bibitem[Ibragimov and Linnik(1971)]{ibragimov1971independent}
I.A. Ibragimov and I.U.I.U.V. Linnik.
\newblock \emph{Independent and Stationary Sequences of Random Variables}.
\newblock Wolters-Noordhoff., 1971.

\bibitem[Jiang(2004)]{jiang2004}
Tiefeng Jiang.
\newblock The asymptotic distributions of the largest entries of sample
  correlation matrices.
\newblock 14\penalty0 (2):\penalty0 865--880, 05 2004.

\bibitem[Kramer et~al.(2009)Kramer, Eden, Cash, and
  Kolaczyk]{PhysRevE.79.061916}
Mark~A. Kramer, Uri~T. Eden, Sydney~S. Cash, and Eric~D. Kolaczyk.
\newblock Network inference with confidence from multivariate time series.
\newblock \emph{Phys. Rev. E}, 79:\penalty0 061916, 06 2009.

\bibitem[Lahiri(2003)]{lahiri2003resampling}
S.N. Lahiri.
\newblock \emph{Resampling Methods for Dependent Data}.
\newblock Springer, 2003.

\bibitem[Li and Rosalsky(2006)]{Li2006}
Deli Li and Andrew Rosalsky.
\newblock Some strong limit theorems for the largest entries of sample
  correlation matrices.
\newblock \emph{The Annals of Applied Probability}, 16\penalty0 (1):\penalty0
  423--447, 2006.

\bibitem[Liu et~al.(2008)Liu, Lin, and Shao]{liu2008}
Wei-Dong Liu, Zhengyan Lin, and Qi-Man Shao.
\newblock The asymptotic distribution and berry¨cesseen bound of a new test for
  independence in high dimension with an application to stochastic
  optimization.
\newblock 18\penalty0 (6):\penalty0 2337--2366, 12 2008.

\bibitem[Liu and Wu(2010)]{liu2010asymptotics}
Weidong Liu and Wei~Biao Wu.
\newblock Asymptotics of spectral density estimates.
\newblock \emph{Econometric Theory}, 26\penalty0 (4):\penalty0 1218--1245,
  2010.

\bibitem[Nagaev(1979)]{nagaev1979large}
Sergey~V Nagaev.
\newblock Large deviations of sums of independent random variables.
\newblock \emph{The Annals of Probability}, pages 745--789, 1979.

\bibitem[Pinelis(1994)]{pinelis1994optimum}
Iosif Pinelis.
\newblock Optimum bounds for the distributions of martingales in banach spaces.
\newblock \emph{The Annals of Probability}, pages 1679--1706, 1994.

\bibitem[Politis et~al.(1999)Politis, Romano, and Wolf]{politis1999subsampling}
D.N. Politis, J.P. Romano, and M.~Wolf.
\newblock \emph{Subsampling}.
\newblock Springer, 1999.

\bibitem[Portnoy(1986)]{portnoy1986central}
Stephen Portnoy.
\newblock On the central limit theorem in ${R}^p$ when $p\rightarrow \infty$.
\newblock \emph{Probability Theory and Related Fields}, 73\penalty0
  (4):\penalty0 571--583, 1986.

\bibitem[Priestley(1988)]{priestley1988non}
Maurice~Bertram Priestley.
\newblock Non-linear and non-stationary time series analysis.
\newblock 1988.

\bibitem[Priestley(1981)]{priestley1981}
M.B. Priestley.
\newblock \emph{Spectral Analysis and Time Series}.
\newblock Academic Press, 1981.

\bibitem[Rosenblatt(1956)]{rosenblatt1956central}
Murray Rosenblatt.
\newblock A central limit theorem and a strong mixing condition.
\newblock \emph{Proceedings of the National Academy of Sciences of the United
  States of America}, 42\penalty0 (1):\penalty0 43, 1956.

\bibitem[Rosenblatt(1971)]{rosenblatt1971markov}
Murray Rosenblatt.
\newblock \emph{Markov processes: structure and asymptotic behavior}.
\newblock Springer, 1971.

\bibitem[Rosenblatt(1985)]{rosenblatt1985}
Murray Rosenblatt.
\newblock \emph{Stationary sequences and random fields}.
\newblock Springer, 1985.

\bibitem[Tong(1990)]{tong1990non}
Howell Tong.
\newblock \emph{Non-linear time series: a dynamical system approach}.
\newblock Oxford University Press, 1990.

\bibitem[Tsay(2005)]{tsay2005analysis}
Ruey~S Tsay.
\newblock \emph{Analysis of Financial Time Series}, volume 543.
\newblock John Wiley \& Sons, 2005.

\bibitem[Wiener(1958)]{wiener1958nonlinear}
N~Wiener.
\newblock \emph{Nonlinear Problems in Random Theory}.
\newblock Wiley, New York, 1958.

\bibitem[Wu(2005)]{Wu2005}
Wei~Biao Wu.
\newblock Nonlinear system theory: another look at dependence.
\newblock \emph{Proceedings of the National Academy of Sciences of the United
  States of America}, 102\penalty0 (40):\penalty0 pp. 14150--14154, 2005.

\bibitem[Wu(2011)]{wu2011asymptotic}
Wei~Biao Wu.
\newblock Asymptotic theory for stationary processes.
\newblock \emph{Statistics and Its Interface, 0}, pages 1--20, 2011.

\bibitem[Wu and Wu(2015)]{wu2014}
Wei~Biao Wu and Ying~Nian Wu.
\newblock High-dimensional linear models with dependent observations.
\newblock \emph{Manuscript}, 2015.

\bibitem[Xiao and Wu(2012)]{xiao2012covariance}
Han Xiao and Wei~Biao Wu.
\newblock Covariance matrix estimation for stationary time series.
\newblock \emph{The Annals of Statistics}, 40\penalty0 (1):\penalty0 466--493,
  2012.

\bibitem[Xiao and Wu(2013)]{Xiao20132899}
Han Xiao and Wei~Biao Wu.
\newblock Asymptotic theory for maximum deviations of sample covariance matrix
  estimates.
\newblock \emph{Stochastic Processes and their Applications}, 123\penalty0
  (7):\penalty0 2899 -- 2920, 2013.

\bibitem[Zhang and Cheng(2014)]{zhang2014}
Xianyang Zhang and Guang Cheng.
\newblock Bootstrapping high dimensional time series.
\newblock \emph{arXiv preprint arXiv:1406.1037}, 2014.

\end{thebibliography}
\end{document}